\documentclass{birkjour}
\usepackage{amscd}
\usepackage{amssymb}
\usepackage{mathrsfs}
 \newtheorem{thm}{Theorem}[section]
 \newtheorem{cor}[thm]{Corollary}
 \newtheorem{lem}[thm]{Lemma}
 
 \theoremstyle{definition}
 
 \theoremstyle{remark}

 \numberwithin{equation}{section}
\usepackage[final]{hyperref}
\hypersetup{unicode= false, colorlinks=true, linkcolor=blue,
anchorcolor=blue, citecolor=red, filecolor=red, menucolor=blue,
pagecolor=blue, urlcolor=blue}
\usepackage[none]{hyphenat}
\renewcommand{\theenumi}{\roman{enumi}}
\renewcommand{\labelenumi}{(\theenumi)}

\newcommand{\N}{\mathbb{N}}

\newcommand{\CC}{\mathbb {C}}

\title[Carleson type measures for   Fock--Sobolev  spaces] {Carleson type measures for  Fock--Sobolev  spaces}
\author [Tesfa  Mengestie]{Tesfa  Mengestie }
\address{Department of Mathematical Sciences\\
Norwegian University of Science and Technology (NTNU)\\
 NO- 7491 Trondheim, Norway}
\email{tesfantnu@gmail.com}
\begin{document}
\begin{abstract}
We describe  the $(p,q)$ Fock--Carleson measures for weighted Fock--Sobolev spaces  in terms of the objects
  $(s,t)$-Berezin transforms, averaging functions, and averaging sequences on the complex  space $\CC^n$. The main results show that while these objects  may have growth not faster than  polynomials  to induce the $(p,q)$ measures for $q\geq p$, they should be of $L^{p/(p-q)}$ integrable against a weight of polynomial growth for $q<p$.
   As an application, we characterize the bounded and compact weighted composition operators on  the Fock--Sobolev spaces in terms of certain Berezin type integral transforms on  $\CC^n$. We also obtained estimation results for  the norms and essential norms of  the operators in terms of the integral transforms.  The results obtained   unify and extend a number of other results in the area.
\end{abstract}
 \maketitle
\section{Introduction}
The classical weighted
Fock space $\mathcal{F}_\alpha^p$ consists of entire functions $f$ on $\CC^n$ for which
\begin{equation*}
\|f\|_{p}^p=  \Big(\frac{\alpha p}{2\pi}\Big)^n\int_{\CC^n}
|f(z)|^p e^{-\frac{\alpha p}{2}|z|^2} dV(z) <\infty
\end{equation*} where   $dV$  denotes the usual Lebesgue measure
on $\CC^n$, $ 0 < p <\infty$,  and $\alpha$ is a positive parameter. For $p= \infty$, the corresponding space consists of all such  $f$'s  for which
\begin{equation*}
\|f\|_{\infty}= \sup_{z\in \CC^n}
|f(z)|e^{-\frac{\alpha}{2}|z|^2} <\infty.
\end{equation*}
   The space $\mathcal{F}_\alpha^2$, in particular, is a reproducing kernel Hilbert
space with  kernel  and normalized reproducing kernel  functions respectively  given by
$K_{w}(z)= e^{\alpha \langle z, w\rangle}$ and
$k_{w}(z)= e^{\alpha\langle z, w\rangle-\alpha|w|^2/2}$ where
\begin{align*}\langle z, w\rangle= \sum_{j= 1}^n z_j\overline{w_j},\ \ |z|= \sqrt{\langle z, z\rangle},\ w=(w_j), z=(z_j)\in \CC^n.  \end{align*}
For an  $n$-tuple $\beta= (\beta_1, \beta_2, . . ., \beta_n)$  of nonnegative integers we also write
$ \partial^\beta=\partial_1^{\beta_1}...\partial_n^{\beta_n} $ where $\partial_j$ denotes partial differentiation  with respect to
the j-th component. For any non-negative integer $m$ and $0< p \leq \infty$, the weighted Fock--Sobolev spaces $\mathcal{F}_{(m,\alpha)}^p$  of order $m$  consists of
entire functions $f$ on $\CC^n$ such that
\begin{equation}\label{norm}
\|f\|_{(p, m)}= \sum_{\beta_{sn}\leq m} \|\partial^\beta f\|_{p} <\infty
\end{equation} where $\beta_{sn}=\beta_1+\beta_2+...+\beta_n.$ These spaces have recently been
introduced by R. Cho  and  K. Zhu \cite{RCKZ}, and one of their main results  provides the following  useful
Fourier characterization of the spaces.
\begin{lem}\label{Fourier} Let $0<p\leq \infty$. Then an entire function $f$ on $\CC^n$ belongs to   $\mathcal{F}_{(m,\alpha)}^p$ if and only if $z^\beta f$ belongs to  $\mathcal{F}_{\alpha}^p$  for all nonnegative  multi-indices $\beta$
with $\beta_{sn}= m$  where  $z^\beta= z_1^{\beta_1} z_2^{\beta_2}z_3^{\beta_3}... z_n^{\beta_n}$ .
\end{lem}
As a consequence of this lemma, it was proved that the norm in \eqref{norm} is comparable to
 \begin{equation}
 \label{normequal}
\|f\|_{(p, m)}= \Big(C_{(p,m,n)} \int_{\CC^n}
|z|^{mp} |f(z)|^p e^{-\frac{\alpha p}{2}|z|^2} dV(z)\Big)^{1/p}
\end{equation}
for $ 0 < p <\infty$ and
\begin{equation*}
C_{(p,m,n)}= \Big(\frac{\alpha p}{2}\Big)^{(mp/2)+ n} \frac{\Gamma(n)}{
\pi^n \Gamma Γ((mp)/2 + n)},
\end{equation*} where $ \Gamma$ denotes the Gamma function. For $p= \infty,$ the corresponding norm is
\begin{equation}
\label{normeqqual2}
\|f\|_{(\infty,m)}= \sup_{z\in \CC^n}
|z|^m|f(z)|e^{-\frac{\alpha}{2}|z|^2}.
\end{equation}  We  find it more convenient to use
   this equivalent norm through out the rest of the  paper. We  note in passing that
   the Fock--Sobolev spaces of order $m$ can also  be considered  as a weighted (generalized) Fock spaces $\mathcal{F}_{\varphi_{m}}^p$  consisting  of
   entire functions $f$  for which
  \begin{align*}
  \label{weight} \Big(\frac{\alpha p}{2\pi}\Big)^n\int_{\CC^n}
|f(z)|^p e^{-p \varphi_{m}(z)} dV(z) <\infty
\end{align*} for $0<p<\infty$ and $\sup_{z\in \CC^n} |f(z)|e^{-\varphi_{m}(z)} <\infty$ for $p=\infty $ where
$\varphi_{m}(z)= - m\log (1+|z|)+ \alpha |z|^2/2.$

We  next   recall  the notion of
lattice for the  space $\CC^n$. For a positive $r$ we denote by $D(z,r)$ the  set $\{ w\in\CC^n: |z-w|<r\}.$ We  say that a sequence of distinct points
$(z_k)_{k\in \N} \subset \CC^n$ is  an $r/2-$ lattice  for  $\CC^n$ if the  sequence of the balls $D(z_k, r), \ k\in \N $ constitutes a covering of $\CC^n$ and the balls $ D(z_k, r/2)$ are mutually disjoint.  The
sequence $(z_k), k\in \N$ will refer to such  $r/2$ lattice with a fixed $r$ in the
remaining part of the paper. An interesting  example of such a lattice  can be found in \cite{ZHXL}.
\begin{lem} \label{covering} Let $r>0$ and $(z_k)_{k\in \N}$ be  an $r/2-$ lattice  for  $\CC^n$. Then there exists a positive integer $N_{\max}$ such that every point in $\CC^n$ belongs to at
most $N_{\max}$ of the balls $D(z_k, 2r)$.
\end{lem}
The proof of the lemma can be found in \cite{KH,KZH} where in
\cite{KH} a more general setting has been considered.

Let $\mu$ be a positive Borel measure on $\CC^n.$ Then its  average on $D(z,r)$ is the quantity
$\mu(D(z,r))/Vol(D(z,r))$ where $Vol(D(z,r))$ is the Euclidean volume of the
ball which is a constant for all $z$ in $\CC^n.$ In what follows, we simply refer
$\mu(D(.,r))$  as an averaging function of $\mu,$ and $\mu(D(z_k,
r))$ as its averaging sequence.

A word on notation: The notation $U(z)\lesssim
V(z)$ (or equivalently $V(z)\gtrsim U(z)$) means that there is a
constant $C$ such that $U(z)\leq CV(z)$ holds for all $z$ in the set
in question, which may be a Hilbert space or  a set of complex
numbers. We write $U(z)\simeq V(z)$ if both $U(z)\lesssim V(z)$ and
$V(z)\lesssim U(z)$.
\section{ The $(p,q)$ Fock--Carleson measures on Fock--Sobolev spaces }
Carleson measures were first introduced by  L. Carleson \cite{Carl} as a tool  to study interpolating sequences in the Hardy space
$H^\infty $ of bounded analytic functions in the unit disc and  the corona problem. Since then the measures have found numerous applications and extensions  in the study  of  various spaces of functions: for example see \cite{Al1,Ba4, BMS2,SCBS, Co2, Co1,DL1, SCPower,TV}.  In this paper, we study one of its extensions namely the $(p,q)$ Fock--Carleson measures for  weighted Fock--Sobolev spaces. In the
 next section, we will also look at   application of such  measures in studying  some mapping  properties of weighted composition operators  acting between different weighted Fock--Sobolev spaces.

Let $0<p\leq \infty$ and $0<q< \infty$. Then we call  a nonnegative measure $\mu$ on $\CC^n$ a
$(p, q) $ Fock--Carleson measure for  Fock--Sobolev spaces if\footnote{We follow the approach to Carleson measures
 taken in \cite{RCKZ}.}
\begin{equation}
\label{carleson} \int_{\CC^n} |f(z)|^q e^{\frac{-\alpha
q}{2}|z|^2} d\mu(z)\lesssim \| f\|_{(p,m)}^q
\end{equation} for all $f$ in $\mathcal{F}_{(m,\alpha)}^p$. In other words,  $\mu$ is a $(p,q)$ Fock--Carleson measure if and only if
the canonical embedding map $I_\mu: \mathcal{F}_{(m,\alpha)}^p \to L^q(\sigma_q)$ is bounded where $d\sigma_q(z)= e^{-\frac{q\alpha}{2} |z|^2}d\mu(z)$.  We
 call $\mu$ a $(p, q) $ vanishing Fock--Carleson measure if
\begin{equation*}
\lim_{j\to \infty}\int_{\CC^n} |f_j(z)|^q
e^{\frac{-q\alpha}{2}|z|^2} d\mu(z)= 0
\end{equation*} whenever $f_j$ is a uniformly bounded sequence in
$\mathcal{F}_{(m,\alpha)}^p$ that converges uniformly to zero on compact subsets of
$\CC^n$ as $j \to \infty.$ We  will write $\|\mu\|= \|I_\mu\|$ for the smallest
admissible constant in inequality \eqref{carleson} which often is called the Carleson constant.

For $s, t>0, $ we may  define the $(t,s)$-Berezin  type transform of $\mu$
by
  \begin{equation*}
\widetilde{\mu}_{(t,s)}(w)= \int_{\CC^n} (1+|z|)|^{-s} e^{-\frac{t\alpha}{2}|z-w|^2}d\mu(z).
\end{equation*} As will be seen, its role is analogous to that played by the Berezin transform for the Bergman spaces.
For convenience, we will also  use the  notations
 \begin{align*}
 \mu_s(z)= \frac{\mu(z)}{(1+|z|)^{s}}, \ \mu_{(s, r, D)}(z)=\frac{\mu(D(z,r))}{(1+|z|)^{s}}, \ \text{and}\ \ L^p= L^p(\CC^n, dV).\end{align*}   We may now  state our first main result.
\begin{thm}\label{carlesonmeasure}
Let $0<p\leq q< \infty$ and $\mu \geq 0.$ Then the following statements are
equivalent.
\begin{enumerate}
\item $\mu $ is  a $(p, q)$ Fock--Carleson measure;
\item $\widetilde{\mu}_{(t,mq)}\in L^\infty$ for some (or any)  $t>0;$
\item $\mu_{(mq,r,D)} \in L^\infty$ for some (or any)  $r>0;$
\item $\mu_{(mq,r,D)}(z_k) \in \ell^\infty$ for some (or any)  $r>0.$    Moreover,  we  have
   \begin{equation}
   \label{normestimate1}
   \|\mu\|^q \simeq
\|\widetilde{\mu}_{(t, mq)}\|_{L^\infty} \simeq \|\mu_{(mq, r, D)} \|_{
L^\infty}\simeq \| \mu_{(mq, r, D)}(z_k) \|_{\ell^{\infty}}.
   \end{equation}
\end{enumerate}
\end{thm}
Vanishing  Carleson measures appear naturally in the study of compact
composition operators, Toeplitz and Hankel operators,  Volterra type integral operators,
two weight Hilbert transforms, and in several  other contexts
in various functional spaces. As far as  their
characterization is concerned, there  exists a  general ``folk
theorem'':
 once the Carleson measures are described by a certain ``big oh''
condition, vanishing Carleson measures are then characterized by the
corresponding ``little oh'' counterparts. This does not however mean that
such  `` folk theorem'' is always true. See  \cite{FMB} for a
counterexample. Our next result shows that  it  still holds on Fock--Sobolev spaces.
\begin{thm}\label{compactcarleson}
Let $0<p\leq q< \infty$ and $\mu \geq 0.$ Then the following statements are
equivalent.
\begin{enumerate}
\item $\mu $ is a $(p, q)$ vanishing Fock--Carleson measure;
\item $\widetilde{\mu}_{(t, mq)}(z)\to 0 \ \ \text{as} \ \ |z| \to \infty$ for some (or any)  $t>0;$
\item $\mu_{(mq,r,D)}(z)\to 0 \ \ \text{as} \ \ |z| \to \infty$ for some (or any)  $r>0;$
\item  $\mu_{(mq,r, D)}(z_k) \to 0\ \ \text{as} \ \ k \to \infty$ for some (or any)  $r>0.$
\end{enumerate}
\end{thm}
 Conditions (ii), (iii) and (iv) in the two theorems above are independent of the parameter $\alpha$ and exponent $p\leq q$. It means that
 if $\mu$ is a $(p, q)$  (vanishing) Fock--Carleson measure for some $p\leq q$ and $\alpha >0,$ then it is a $(p_1, q)$ (vanishing) Fock--Carleson measure for any $p_1 \leq q$ and every other parameter $\alpha.$ On the other hand, the conditions are dependent on the size of the exponent $q$ in the target space in the sense that if $\mu$ constitutes a (p, q) Fock--Carleson measure for some $q\geq p,$ then it may fail to be a $(p, q_1)$ Fock measures for any $q_1\geq p$ unless $m=0$ or $q_1\geq q.$ This presents a clear  distinction  with the corresponding conditions for  the ordinary Fock spaces $(m=0).$ Because, in the later, it holds that $\mu$ is a $(p,q)$ Fock--Carleson measure for some $p\leq q$ if and only if  it  is a $(p_1,q_1)$ Fock--Carleson measure for  any pair  of exponents  $(p_1,q_1)$  for which $p_1\leq q_1$.  If we take a different approach to the $(p,q)$ measures  and  redefine inequality  \eqref{carleson} by replacing $d\mu(z)$ with  $(1+|z|)^{mq} d \mu(z)$ the distinction mentioned above  would  disappear and the (p,q) measure conditions will be exactly the same as they are for ordinary Fock spaces.

As in the case of ordinary Fock spaces, the  Fock--Sobolev spaces  satisfy  the inclusion monotonicity  property  $\mathcal{F}_{(m,\alpha)}^p \subseteq \mathcal{F}_{(m,\alpha)}^q $  whenever $0 <p\leq q \leq \infty$ \cite{RCKZ}.  Thus,  for  $p>q,$  the boundedness conditions on  the averaging functions, averaging sequences
  and $(t, mq)$-Berezin transforms  will be  replaced by  the next stronger  $p/(p-q)$ integrability  against a weight of polynomial growth conditions.
\begin{thm}\label{strongcarlesonmeasure}
Let $0<q <p<\infty$ and $\mu \geq 0.$ Then the following statements are
equivalent.
\begin{enumerate}
\item $\mu $ is  a $(p, q)$ Fock--Carleson measure;
\item $\mu $ is a  $(p, q)$ vanishing Fock--Carleson measure;
\item $\widetilde{\mu}_{(t, mq)}\in L^{\frac{p}{p-q}}$ for some (any) $t>0$;
\item $\mu_{(mq,r, D)}\in L^{\frac{p}{p-q}} $ for some (or any)  $r>0;$
\item $\mu_{(mq,r,D)}(z_k)\in \ell^{\frac{p}{p-q}} $ for some (or any)  $r>0.$  Moreover, we have
\begin{equation}\label{normestimate2}
\|\mu\|^q \simeq \|\mu_{(mq, r, D)}\|_{L^{\frac{p}{p-q}}}\simeq \|\widetilde{\mu}_{(t,mq)}\|_{L^{\frac{p}{p-q}}}\simeq \| \mu_{(mq, r, D)}(z_k)\|_{\ell^{\frac{p}{p-q}}}.
\end{equation}
\end{enumerate}
\end{thm}
Observe that the fraction $p/(p-q)$ is the conjugate exponent of $p/q$ whenever $0<q\leq p<\infty.$ In the limiting case, i.e., when  $ p= \infty,$  the next
 yet stronger condition holds.
\begin{thm}\label{boundedinfinity}
Let $0<q< \infty$ and $\mu \geq 0.$ Then the following statements are equivalent.
\begin{enumerate}
\item $\mu$ is an $(\infty, q)$ Fock--Carleson measure;
\item $\mu$ is an $(\infty, q)$ vanishing Fock--Carleson measure;
\item $\widetilde{\mu}_{(t,mq)}\in L^1$ for some(or any) $t>0$;
\item $\mu_{(mq,r,D)}\in L^1$ for some (or any)  $r>0;$
\item $\mu_{(mq,r,D)}(z_k)\in \ell^1 $ for some (or any)  $r>0;$
\item $\mu_{mq}$ is a finite measure on $\CC^n.$ Moreover,  we  have
   \begin{equation}
   \label{fourequal}
   \|\mu\|^q \simeq
\|\widetilde{\mu}_{(t, mq)}\|_{L^1} \simeq \|\mu_{(mq,r,D)}\|_{
L^1}\simeq \mu_{mq} (\CC^n)\simeq \|\mu_{(mq,r,D)}(z_k)\|_{\ell^1}.
   \end{equation}
\end{enumerate}
\end{thm}
 The four  results above unify and extend a number of recent results in
the area. For example when $m=0$,  while  the first three of the   results  simplify to   results obtained
in \cite{ZHXL}, Theorem~\ref{boundedinfinity} simplifies  to  a  result obtained in \cite{TM}. On the other hand, when $p=q$ the first two theorems give  Theorem~21 and Theorem~22 of \cite{RCKZ}. If $m=0$ and $p=q=2,$ then the first two theorems again simplify to
results obtained in  \cite{JIKZ}.
\section{ Weighted composition operators on Fock--Sobolev spaces}
Let $H(\CC^n)$ denotes the space of entire functions on $\CC^n.$ Each pair of entire functions $(\psi, u)$  induces a weighted composition operator $uC_\psi f = u (f\circ\psi )$  on
$H(\CC^n)$.  Questions about boundedness, compactness, and other operator theoretic properties
of $uC_\psi$ expressed in terms of function theoretic conditions on $u$ and $\psi$ have been a subject
of high interest,  and  have been studied by several authors  in various function spaces. The Schatten class
membership properties of $uC_\psi$ on $\mathcal{F}_{(m,\alpha)}^2$ has recently been studied in \cite{TM1}. In this section, we will study  the bounded and compact mapping properties of $uC_\psi$ when it acts between different  weighted Fock--Sobolev spaces. We will also estimate the norm and essential norm of $uC_\psi$ in terms of certain Berezin type integral  transforms.   The approach we intend to follow links some of these properties  of  $uC_\psi$ with the $(p,q)$ Fock--Carleson measures which allows us to apply   the results obtained in the previous section.  Indeed, this  offers  a simple example where the $(p,q)$ Fock--Carleson   measures   find some applications in operator theory.

 Our  results  on $uC_\psi$ will be  expressed in terms of the  function
 \begin{eqnarray*}
 B^\infty_{(m, \psi)} (|u|)(z)
 = \frac{|z|^m |u(z)|}{(1+|\psi(z)|)^m} e^{\frac{\alpha}{2}\big(|\psi(z)|^2-|z|^2\big)}
 \end{eqnarray*} and a Berezin type integral transform
 \begin{align*}
B_{(m,\psi)}(|u|^p)(w)= \int_{\CC^n} \frac{|k_w(\psi(z))|^p}{(1+|\psi(z)|)^{mp}} |u(z)|^p|z|^{pm} e^{-\frac{\alpha
p}{2}|z|^2}dV(z).
\end{align*}
\subsection{Bounded and compact $uC_\psi$}
\begin{thm}\label{bounded}
Let $0<p\leq q<\infty$ and $(u,\psi)$ be a pair of entire functions.  Then
$uC_{\psi}: \mathcal{F}_{(m,\alpha)}^p \to \mathcal{F}_{(m,\alpha)}^q$ is
\begin{enumerate}
 \item bounded if and only if $ B_{(m,\psi)}(|u|^q)$ belongs to $ L^\infty$.
Moreover, we have
\begin{equation}
\label{norm1estimate} \|uC_\psi\| \simeq  \|B_{(m,\psi)}(|u|^q)\|_{L^\infty}^{1/q}.
\end{equation}
\item compact if and only if $$\lim_{|z|\to \infty}
B_{(m, \psi)}(|u|^q) (z)= 0.$$
\end{enumerate}
\end{thm}
Note that like in Theorem~\ref{carlesonmeasure},  the conditions both in (i) and (ii) are independent of the
 exponent $p$ apart from the fact that $p$ should not be exceeding $q$. In other words,
  if there exists a $p>0$ for which  $ uC_\psi$ is bounded (compact) from $ \mathcal{F}_{(m,\alpha)}^p $ to $ \mathcal{F}_{(m,\alpha)}^q$, then it is also  bounded (compact) from $ \mathcal{F}_{(m,\alpha)}^{p_1} $
  to $ \mathcal{F}_{(m,\alpha)}^q$ for every other $p_1 \leq q$.

A natural question is whether there exists an interplay between the
two symbols  $u$ and $\psi$  in inducing bounded and compact
operators $uC_\psi$. We first observe that by the classical
 Liouville's theorem   a nonconstant function $u$ can not decay. The following is a simple consequence of
 this fact.
\begin{cor}\label{cor3}
Let $0<p\leq q<\infty$ and $(u,\psi)$ be a pair of entire functions. If $u \neq 0$ and
$uC_{\psi}: \mathcal{F}_{(m,\alpha)}^p \to \mathcal{F}_{(m,\alpha)}^q$ is
 bounded, then $\psi(z)= Az+B$ where $A$ is an $n\times n$ matrix, $\|A\|\leq 1$ and $B$ is an $ n\times 1$ matrix such that $\langle Aw, B\rangle=0$ whenever  $|Aw|=|w|$ for some $w \in \CC^n.$ Moreover, if  $uC_{\psi}$ is compact, then $\|A\|< 1$ where
 $ \|A\|$ refers to the operator norm of matrix $A$.
\end{cor}
 By setting $u=1$ and simplifying the conditions in Theorem~\ref{bounded}, one can easily see that the linear   forms for $\psi$
  are both necessary and sufficient  for $C_{\psi}: \mathcal{F}_{(m,\alpha)}^p \to \mathcal{F}_{(m,\alpha)}^q$ to be bounded (compact). This fact together with Corollary~\ref{cor3} ensures that boundedness  of $uC_\psi$ implies boundedness of
 $C_\psi$ while the converse in general fails.  The same conclusion holds for compactness. Particular cases of theses conclusions  could be also read  in \cite{CMS,CCK}.
\begin{proof}
Observe that
\begin{align*}
B_{(m, \psi)}(|u|^q) (z)&\geq \int_{\CC^n} \frac{|k_w(\psi(z))|^q}{(1+|\psi(z)|)^{mq}} |u(z)|^q|z|^{qm} e^{-\frac{\alpha
q}{2}|z|^2}dV(z)\\
&\gtrsim  \frac{|k_w(\psi(z))|^q}{(1+|\psi(z|)^{mq}} |u(w)|^q|z|^{qm} e^{-\frac{\alpha
q}{2}|z|^2}
\end{align*} for all $z,w \in \CC^n.$ Applying Theorem~\ref{bounded} and setting $w= \psi(z)$ in particular  gives
\begin{align}
\label{forcont}
\infty > \sup_{w\in \CC^n}\frac{|k_w(\psi(z))|^q|u(z)|^q|z|^{qm}}{(1+|\psi(z|)^{mq}e^{\frac{\alpha
q}{2}|w|^2}}  \geq \frac{|u(z)|^q|z|^{qm}}{(1+|\psi(z)|)^{mq}} e^{\frac{\alpha
q}{2}|\psi(z)|^2-|z|^2}.
\end{align} Indeed, we claim that
\begin{align}
\label{claim}
\limsup_{|z|\to \infty} \frac{|\psi(z)|}{|z|} \leq 1.
\end{align} We argue by contradiction, and suppose \eqref{claim} fails. Then there exists a sequence $(z_j)$ such that
$|z_j|\to \infty $ as $j\to \infty$ and $$\limsup_{|z_j|\to \infty} |\psi(z_j)|/|z_j| >1.$$ 
For nobility,  we set $w_j= |\psi(z_j)|/|z_j|,$ and observe 
\begin{align}
\label{contradiction}
\limsup_{j\to \infty} M_{\infty}(u^q,(1+|\psi(z_j)|)^{mq}) \lesssim \limsup_{j\to \infty}\frac{1}{|z_j|^{qm}e^{\frac{\alpha
q}{2}(|\psi(z_j)|^2-|z_j|^2)}}\nonumber\\
=\limsup_{j\to \infty}
 \frac{1 }{|z_j|^{mq}e^{\frac{\alpha q}{2}|z_j|^2(w_j^2-1)}}=0
\end{align} which gives a contradiction  as   $u$ is a constant entire function and $M_{\infty}(u^q,(1+|\psi(z_j)|)^{mq})$ is 
the integral mean of the function $|u|^q$. Thus, \eqref{claim} implies $\psi(z)= Az+ B$  for some  $A$ an $n\times n$ matrix with $\|A\|<1$ and $B\in \CC^n.$

Let now $\eta$ be a point in $\CC^n$ such that $|A\eta|= |\eta|$. We may further assume that $|\eta|=1$ and $A\eta= \eta$ where the latter is due to unitary change of variables; see the proof  of \cite[Theorem~1]{CMS}. If $z= t\tau \eta$ where $ |\tau|=1$  is  a constant for which $\tau\langle A\eta, B\rangle= |\langle A\eta, B\rangle|$, then
\begin{align}
\label{frac}
\frac{|u(z)|^q|z|^{qm}}{(1+|\psi(z)|)^{mq}} e^{\frac{\alpha
q}{2}|\psi(z)|^2-|z|^2}=\frac{|u(t\tau\eta)|e^{\frac{\alpha q}{2}( |B|^2+ 2t|\langle A\eta,B\rangle|)}}{1+t^{-2}|B|^2 + 2t^{-1}|\langle A\eta,B\rangle|}.
\end{align} By \eqref{forcont}, the fraction \eqref{frac} has to be finite as $t\to \infty,$ and this holds only if $\langle A\eta,B\rangle=0$
as desired.

If, in addition, $uC_\psi$ is compact, then by part (ii) of Theorem~\ref{bounded},
\begin{align}
\label{comcom}
\lim_{|z|\to \infty}\frac{|u(z)|^q|z|^{qm}}{(1+|\psi(z)|)^{mq}} e^{\frac{\alpha
q}{2}|\psi(z)|-|z|^2}=0.
\end{align} A simple modification of the above arguments show that \eqref{comcom} holds only if
$\psi(z)= Az+ B$ with $\|A\|<1.$
\end{proof}
We now consider the case where $p>q$. Mapping $\mathcal{F}_{(m,\alpha)}^p$ into $ \mathcal{F}_{(m,\alpha)}^q$
gives the following  stronger integrability  conditions as one would expect.
\begin{thm} \label{bounded1} Let $0<q<p<\infty$ and $(u,\psi)$ be a pair of entire functions.  Then
the following statements are equivalent.
\begin{enumerate}
 \item $uC_{\psi}: \mathcal{F}_{(m,\alpha)}^p \to \mathcal{F}_{(m,\alpha)}^q$ is bounded;
  \item $uC_{\psi}: \mathcal{F}_{(m,\alpha)}^p \to \mathcal{F}_{(m,\alpha)}^q$ is  compact;
  \item  $B_{(m,\psi)}(|u|^q)$ belongs to $ L^{\frac{p}{p-q}}$.  We further have the norm estimate
\begin{equation}
\label{normless1} \|uC_\psi\| \simeq \|
B_{(m, \psi)}(|u|^q)\|_{L^{\frac{p}{p-q}}}^{\frac{p-q}{p}}.
\end{equation}
\end{enumerate}
\end{thm}
Following a similar approach as in the proof of Corollary~\ref{cor3}, we observe that the  $L^{p/(p-q)}$ integrability of $B_{(m,\psi)}(|u|^q)$  restricts  further   $\psi$ to have only the linear  form
$ \psi(z)= Az+B$  with  $\|A\|<1.$
 \begin{thm}\label{bounded3}
Let $0<q< \infty$ and $(u,\psi)$ be a pair of entire functions.   Then  the
following  statements are  equivalent.
\begin{enumerate}
\item $uC_{\psi}: \mathcal{F}_{(m,\alpha)}^\infty \to \mathcal{F}_{(m,\alpha)}^q$ is
  bounded;
  \item $uC_{\psi}: \mathcal{F}_{(m,\alpha)}^\infty \to \mathcal{F}_{(m,\alpha)}^q$ is compact;
  \item  $B_{(m, \psi)} (|u|^q)$ belongs to $L^1.$  Furthermore, we have the asymptotic norm estimate
  \begin{equation}
  \label{normless2}
  \|uC_\psi\| \simeq \|B_{(m,\psi)}
  (|u|^q)\|_{L^1}^{1/q}.
  \end{equation}
  \end{enumerate}
 \end{thm}
 As it  will be seen  in the proofs, the boundedness and compactness  conditions for $uC_\psi$  in Theorems~\ref{bounded}-\ref{bounded3}
 can be equivalently expressed in terms of $(p,q)$ (vanishing) Fock--Carleson measures, averaging functions, and  averaging sequences of appropriately chosen positive measures $\mu$ on $\CC^n$.  .
 \begin{thm}\label{bounded2}
 Let $0<p\leq \infty$ and $(u,\psi)$ be a pair of entire functions.     Then $uC_{\psi}: \mathcal{F}_{(m,\alpha)}^p \to \mathcal{F}_{(m,\alpha)}^\infty$ is
 \begin{enumerate}
 \item   bounded if and only if  $B^\infty_{(m, \psi)} (|u|)$ belongs to $ L^\infty.$  Moreover, we have
\begin{equation}
\label{normbound} \|uC_\psi\|\simeq \|
B^\infty_{(m,\psi)} (|u|)\|_{L^\infty}.
\end{equation}
\item  compact if and only if it is bounded and
 \begin{equation}
 \label{normcomp}
 \lim_{|\psi(z)|\to \infty} B^\infty_{(m,\psi)} (|u|)(z)=
 0.\end{equation}
\end{enumerate}
 \end{thm}
  An interesting observation is to replace $\mathcal{F}_{(m,\alpha)}^\infty$ by a
smaller space $\mathcal{F}_{(0,m,\alpha)}^\infty;$ consisting of all analytic
function $f$ such that \begin{equation*}\lim_{|z| \to \infty}
|f(z)||z|^me^{-\frac{\alpha}{2}|z|^2} = 0.\end{equation*} The space
$\mathcal{F}_{(0,m,\alpha)}^\infty$ constitutes a proper Banach subspace of
$\mathcal{F}_{(m,\alpha)}^\infty$ and contains  the   spaces  $\mathcal{F}_{(m,\alpha)}^p$ for
all $p<\infty$. If we replace $\mathcal{F}_{(m,\alpha)}^p$  with  a larger space
$\mathcal{F}_{(0,m,\alpha)}^\infty$ in part (i) of the preceding theorem, the condition remains
unchanged. On the other hand, modifying the arguments used to prove part
(ii) of the theorem shows that the following stronger condition holds
when we replace the target space  with $\mathcal{F}_{(0,m,\alpha)}^\infty$.
\begin{cor}
\label{cor2}Let $0<p< \infty$ and $(u,\psi)$ be a pair of entire functions. Then the  map $uC_{\psi}:
\mathcal{F}_{(m,\alpha)}^p  (\ \text{or}\ \mathcal{F}_{(0,m,\alpha)}^\infty) \to
\mathcal{F}_{(0,m,\alpha)}^\infty$ is compact if and only if
\begin{equation}
\label{limit} \lim_{|z|\to \infty} B^\infty_{(m,\psi)}
(|u|)(z)= 0.\end{equation}
\end{cor}
We may mention that for the special case $m=0,$ Theorem~\ref{bounded2} and  its corollary  were proved in
\cite{SS}.
\subsection{Essential norm of $uC_\psi$}
Let $\mathscr{H}_1$ and $\mathscr{H}_2$ be Banach spaces. Then the
 essential norm
$\|T\|_e$ of a bounded operator $T:\mathscr{H}_1 \to \mathscr{H}_2 $
is defined as the distance from $T$ to the space of compact
operators  from  $\mathscr{H}_1$ and $\mathscr{H}_2.$  We refer  readers to \cite{ZZHH, ZZH,TM0, Shapiro,
 SS, UEKI, UEKI2, DV} for estimations of such norms  for  different operators on
Hardy space, Bergman space, $L^p$ and some Fock spaces. We get the
following estimate for $uC_\psi$ when it acts on weighted Fock--Sobolev spaces.
\begin{thm}\label{essentialnorm}
Let $1 <p\leq q \leq \infty$  and $p\neq \infty$. If
$uC_\psi:\mathcal{F}_\alpha^p \to \mathcal{F}_\alpha^q$ is bounded,  then
\begin{equation}
\label{essential} \|uC_\psi\|_e\simeq \begin{cases}
\bigg(\limsup_{|w| \to
\infty} B_{(\psi,m)}(|u|^q)(w) \bigg)^{\frac{1}{q}}, & q<\infty\\
 \limsup_{|\psi(w)| \to
\infty} B^\infty_{(\psi,m)}(|u|)(w), & q= \infty.
\end{cases}
\end{equation}
\end{thm}
For $p>1,$ the compactness conditions in Theorem~\ref{bounded} and Theorem~\ref{bounded2} could
be easily drawn from this relation  since the left-hand side
expression in  \eqref{essential} in this case  vanishes for compact $uC_\psi.$

All the  results obtained on $uC_\psi$  again unify  and  generalize a number of recent results in the area including from \cite{CMS,CCK, TM0, TM, SS, UEKI, UEKI2}. One may simply set $m=0$  and simplify the conditions  to get  the classical  known results on the ordinary Fock spaces.

 We also  mention that we have not explicitly used the kernel function $K_{(w,m)}$ for the space $\mathcal{F}_{(m,\alpha)}^2, \ \ m\neq 0$  in  dealing with any of the results  presented here. Finding an explicit expression for $K_{(w,m)}$ is still an open problem.  By  Corollary~13 of \cite{RCKZ} we have in moduli that
 \begin{align*}
 |K_{(w,m)}(z)| \lesssim \frac{\|K_z\|_{2}\|K_w\|_{2}}{ e^{\frac{\alpha }{8}|z-w|^2} (1+|z||w|)^m}.
 \end{align*} It remains open whether  the  reverse estimate above holds. On the other hand, it was proved in \cite{CCK} that
 \begin{align*}
 \|K_{(w,m)}\|^2_{(m,2)} \simeq \frac{\|K_w\|_{2}^2}{(1+|w|)^{2m}}.
 \end{align*}
\section{Some auxiliary lemmas}
In this section we prove some lemmas which play  key roles in our next considerations. The lemmas are also  interest of their own.
 For  a given
measurable function $f$  and a Borel measure $\mu_f$ on $\CC^n$ such that  $d\mu_f (z)=f(z)dV(z)$, we prove the following.
  \begin{lem}\label{gen} let $1\leq p\leq \infty$  and $0<r, t, s <\infty$. Then  the maps
  $f \mapsto f_{(r,s,D)}$ and  $f \mapsto \widetilde{f}_{(t,s)}$ are bounded on $L^p$ where $f_{(r,s,D)}(z):= (1+|z|)^{-s} \mu_f(D(z,r))$ and \begin{align*}\widetilde{f}_{(t,s)}(z):= \int_{\CC^n} (1+|w|)^{-s}  e^{-\frac{\alpha t}{2}|w-z|^2} d\mu_f(w).\end{align*}
    \end{lem}
 \begin{proof} We mention that for the  case when $s=0$, the lemma was first proved in \cite{ZHXL}. Using the additional fact that
\begin{align}
\label{additional}
\sup_{w\in \CC^n} (1+|w|)^{-s}\leq 1,
\end{align} for all nonnegative $s$ and $t$,
the arguments there can be easily adopted to work for  all positive $s$. We  outline the proof as follows.  We use  interpolation argument on  $L^p$  Lebesgue spaces. Thus, it suffices to establish the statements for $p=1$ and $p=\infty.$ We begin with the case $p=1$.  Using \eqref{additional} and Fubini's theorem, we have
  \begin{align*}
  \|\widetilde{f}_{(t,s)}\|_{L^1}&= \int_{\CC^n} \bigg|\int_{\CC^n}  (1+|w|)^{-s}  e^{-\frac{\alpha t}{2} |w-z|^2}d\mu_f(w) \bigg| dV(z)\nonumber\\
  &=\int_{\CC^n} \bigg(\int_{\CC^n} e^{-\frac{\alpha t}{2} |w-z|^2} dV(z) \bigg)\frac{ |f(w)|}{(1+|w|)^{s}}dV(w)\nonumber\\
   &\simeq \int_{\CC^n}  \frac{ |f(w)|}{(1+|w|)^{s}}dV(w)\lesssim  \|f\|_{L^1}.
  \end{align*} Applying again \eqref{additional} for $t=1$, Fubini's theorem, and the fact that
  $\chi_{D(\zeta,r)}(z)= \chi_{D(z,r)}(\zeta)$ for all $\zeta$ and $z$ in $\CC^n$, we have
  \begin{align*}
  \|f_{(r,s,D)}\|_{L^1}=\int_{\CC^n} (1+|z|)^{-s} \mu_f(D(z,r))dV(z) \leq  \int_{\CC^n} \int_{D(z,r)} |f(\zeta)|dV(\zeta) dV(z)\nonumber\\
  =\int_{\CC^n} |f(\zeta)| \int_{\CC^n} \chi_{D(\zeta,r)}(z) dV(z)dV(\zeta)\simeq \|f\|_{L^1}.
  \end{align*}
   On the other hand, for  $p= \infty$  it easily follows that
  \begin{align*}
  \|f_{(r,s,D)}\|_{L^\infty}= \sup_{z\in \CC^n}\big|(1+|z|)^{-s} \mu_f(D(z,r)) \big| \leq  \sup_{z\in \CC^n}  \int_{D(z,r)} |f(\zeta)|dV(\zeta)\\
  \leq \|f\|_{L^\infty}\sup_{z\in \CC^n}\int_{D(z,r)}dV(\zeta) \lesssim \|f\|_{L^\infty}.
  \end{align*}
  Seemingly,  for  each   $f\in {L^\infty}$,  we also   have
  \begin{align}
  \label{del}
  \|\widetilde{f}_{(t,s)}\|_{L^\infty}&= \sup_{z\in \CC^n}  \int_{\CC^n} e^{-\frac{t\alpha}{2}|w-z|^2}(1+|w|)^{-s} |f(w)|dV(\zeta)\nonumber\\
  &\leq \sup_{z\in \CC^n}  \int_{\CC^n} e^{-\frac{t\alpha}{2}|z-w|^2} |f(\zeta)|dV(w)\nonumber\\
   &\leq \|f\|_{L^\infty} \sup_{z\in \CC^n}  \int_{\CC^n} e^{-\frac{t\alpha}{2}|z-w|^2}dV(w) \lesssim \|f\|_{L^\infty},
     \end{align} and completes the proof.
  \end{proof}
\begin{lem} \label{oneforall}
Let $1\leq p \leq \infty, \ \ 0<\ s<\infty$ and $\mu\geq 0$ be a measure on $\CC^n$. Then  if $\mu_{(s,\delta,D)}$ belongs to $L^p$ for some $\delta>0,$ then  so does $\mu_{(s,r,D)}$  for all $r>0.$
\end{lem}
\begin{proof}
For each $\tau$ in $\CC^n$ we may write
\begin{align*}
\int_{D(\tau, r)} (1+|z|)^{-s} \mu(D(z,\delta)) dV(z)= \int_{D(\tau,r)} \int_{\CC^n}   (1+|z|)^{-s} \chi_{D(z,\delta)}(\zeta) d\mu(\zeta) dV(z).
\end{align*} Using again the simple fact that  $\chi_{D(z,\delta)}(\zeta)= \chi_{D(\zeta,\delta)}(z)$, the double integral
above is easily seen to be equal to
\begin{align*}
\int_{\CC^n} \int_{D(\zeta,\delta) \cap D(\tau, r)} \frac{dV(z) d\mu(\zeta)}{(1+|z|)^{s}} \simeq \int_{\CC^n}\frac{Vol\big(D(\zeta,\delta) \cap D(\tau, r)\big)}{(1+|\tau|)^{s}} d\mu(\zeta)\\
\geq  \frac{1}{(1+|\tau|)^{s}} \int_{D(\tau, r)}Vol\big(D(\zeta,\delta) \cap D(\tau, r)\big) d\mu(\zeta)
\end{align*} where $Vol(E)$ refers to the Lebesque  measure of set $E \subset \CC^n.$   Clearly,
the right hand quantity is bounded from below by
\begin{equation*}
 (1+|\tau|)^{-s} \mu(D(\tau, r)) \inf_{\zeta\in D(\tau, r)} Vol\big(D(\zeta,\delta) \cap D(\tau, r)\big) \gtrsim (1+|\tau|)^{-s} \mu(D(\tau, r))
\end{equation*} where the lower estimate follows here since  $\zeta\in D(\tau, r)$, there obviously  exists a ball $D(\tau_0, r_0)$ contained in $D(\zeta,\delta) \cap D(\tau, r)$  with $Vol\big( D(\tau_0, r_0)\big) \simeq  r_0^n.$ From the above analysis, we conclude
\begin{equation}
\label{ess3}
(1+|\tau|)^{-s} \mu(D(\tau, r)) \lesssim \int_{D(\tau, r)} (1+|z|)^{-s} \mu(D(z,\delta)) dV(z).
\end{equation} If we now  set $f(z)= (1+|z|)^{-s}\mu(D(z,\delta)),  $ then the estimate above  along with Lemma~\ref{gen} ensure
that \begin{equation}\label{forbig}
\|\mu_{(s,r,D)}\|_{L^p} \lesssim \|f_{(s,\delta)}\|_{L^p} \lesssim\|f\|_{L^p}=\|\mu_{(s, \delta, D)}\|_{L^p} <\infty
\end{equation} for each $p\geq1$ and any $r>0.$
\end{proof}
Our next lemma gives the link among averaging sequence, averaging functions
and Berezin type integral transform of a given measure.
 \begin{lem} \label{lemma1}
 Let $1\leq p\leq \infty$ and $0<s<\infty$. Then
 \begin{equation} \label{keylem}
\| \mu_{(s,r, D)}\|_{L^p}\simeq \|\widetilde{\mu}_{(t,s)}\|_{L^p}\simeq  \| \mu_{(s, r, D)}(z_k)\|_{\ell^p}
 \end{equation} for some (or any) $r>0$ and $t>0$.
 \end{lem}
\begin{proof} We begin by noting that  since $\widetilde{\mu}_{(t,s)}$ is independent of  $r,$ if the estimate in \eqref{keylem} holds for
some $r>0,$ then it  holds  for every other positive $r.$  The same holds  with $t$ as  $\mu_{(s, r, D)}$  is independent of  it.
The proof of the lemma follows from a careful modification of  some arguments used  in the proof of Theorem~13 in \cite{JIKZ}. We may first
  observe that for each $w$ in the ball $D(z,r)$, the estimate
  \begin{equation}\label{estimate}
1+|z| \simeq 1+|w|
\end{equation} holds. We proceed to  show the first estimate in \eqref{keylem}. Using \eqref{estimate} we have
 \begin{align*}
 \label{firstt}
 \frac{\mu(D(z,r))}{(1+|z|)^s}= \frac{1}{(1+|z|)^s}\int_{D(z,r)} d\mu(w)\leq\frac{e^{\frac{\alpha t r^2}{2}}}{(1+|z|)^s}\int_{D(z,r)}e^{-\frac{\alpha t }{2} |z-w|^2} d\mu(w)\\
 \lesssim e^{\frac{\alpha t r^2}{2}}\int_{D(z,r)}\frac{e^{-\frac{\alpha t }{2} |z-w|^2}}{(1+|w|)^s} d\mu(w)\lesssim \widetilde{\mu}_{(t,s))}(z)
  \end{align*}  from which we get
  \begin{equation}
  \label{fourth}
   \| \mu_{(s,r,D)}\|_{L^p}\lesssim \|\widetilde{\mu}_{(t,s)}\|_{L^p}
   \end{equation}
   for each $p\geq1.$
On the other hand, by  Lemma~1 of \cite{JIKZ}, we note that the pointwise estimate
\begin{equation}
\label{est}
|f(z)e^{-\frac{\alpha}{2}|z|^2}|^q \lesssim \int_{D(z,r)} |f(w)|^qe^{-\frac{\alpha q}{2}|w|^2} dV(w)
\end{equation} holds for any  $f$ in $ H(\CC^n),$  $q,r>0$ and  $z$ in $\CC^n$.
  From this and \eqref{estimate}, we deduce
\begin{align*}
\frac{|f(z)e^{-\frac{\alpha}{2}|z|^2}|^q }{(1+|z|)^{s}}\lesssim \frac{1}{(1+|z|)^{s}} \int_{D(z,r)} |f(w)e^{-\frac{\alpha}{2}|w|^2}|^q dV(w)\\
\simeq  \int_{D(z,r)} \frac{|f(w)e^{-\frac{\alpha}{2}|w|^2}|^q}{|(1+|w|)^{s}} dV(w)
\end{align*} Integrating the above against the measure $\mu$, we find
\begin{align*}
\int_{\CC^n}\frac{|f(z)e^{-\frac{\alpha}{2}|z|^2}|^q}{(1+|z|)^{s}} d\mu(z) &\lesssim \int_{\CC^n} \int_{D(z,r)} \frac{|f(w)e^{-\frac{\alpha}{2}|w|^2}|^q}{(1+|w|)^{s}} dV(w) d\mu(z)\nonumber\\
&=\int_{\CC^n}  \frac{|f(w)e^{-\frac{\alpha}{2}|w|^2}|^q}{(1+|w|)^{s}} \int_{\CC^n} \chi_{D(w,r)}(z)dV(w) d\mu(z).
\end{align*}
It follows from this estimate and Fubini's theorem  that
\begin{equation}
\label{pointwise}
\int_{\CC^n}\frac{|f(w)e^{-\frac{\alpha}{2}|w|^2}|^q}{(1+|w|)^{s}} d\mu(w)\lesssim \int_{\CC^n}  \big|f(w)e^{-\frac{\alpha}{2}|w|^2}\big|^q \frac{\mu(D(w,r))}{(1+|w|)^{s}}dV(w)
\end{equation} for all entire function $f$ in $\CC^n.$  upon setting $f=k_{z}$  and $q=t$ in it, we  see that the left-hand side integral becomes $\widetilde{\mu}_{(t,s)}$ and
\begin{align*}
\widetilde{\mu}_{(t,s)}(z)&=\int_{\CC^n}(1+|w|)^{-s}|k_{z}(w)|^te^{-\frac{\alpha t}{2}|w|^2} d\mu(w)\nonumber\\
&\lesssim \int_{\CC^n} (1+|w|)^{-s}|k_{z}(w)|^te^{-\frac{\alpha t}{2}|w|^2} \mu(D(w,r))dV(w)
=\widetilde{g}_{(t,s)}(z)
\end{align*} where we set $g(w)= \mu(D(w,r))$. This coupled with Lemma~\ref{gen} yield the
reverse  estimate in \eqref{fourth}. That is
\begin{align}
\|\widetilde{\mu}_{(t,s)}\|_{L^p}\lesssim \|\widetilde{g}_{(t,s)}\|_{L^p} \lesssim \|g\|_{L^p} = \|\mu_{(s,r,D)}\|_{L^p}
\end{align} for all $p.$
 Since the case for $p= \infty$ is trivial, the proof will be complete once we show that
the first and the last quantities in \eqref{keylem} are comparable for $1\leq p<\infty$. In doing so,
\begin{align*}
\int_{\CC^n}\Bigg( \frac{\mu(D(z, r))}{(1+|z|)^s}\Bigg)^p dV(z) &\leq \sum_{k=1}^\infty \int_{D(z_k,r)} \Bigg( \frac{\mu(D(z, r))}{(1+|z|)^s}\Bigg)^p dV(z)\nonumber\\
&\leq\sum_{k=1}^\infty \int_{D(z_k,r)} \Bigg( \frac{\mu(D(z_k, 2r))}{(1+|z_k|)^s}\Bigg)^p dV(z)\nonumber\\
&\lesssim\sum_{k=1}^\infty  \Bigg( \frac{\mu(D(z_k, 2r))}{(1+|z_k|)^s}\Bigg)^p.
\end{align*} Here the last inequality follows since $r$ is fixed and  $Vl(D(z_k,r))\simeq r^n $  independent of
$k$. From this and Lemma~\ref{oneforall} we obtained one side of the required estimate in \eqref{keylem}, namely that
\begin{align}
\label{reverse}
\|\mu_{(s,r, D)}\|_{L^p}\lesssim  \| \mu_{(s, r, D)}(z_k)\|_{\ell^p}.
\end{align}It remains to prove the reverse estimate in \eqref{reverse}. To this end, Observe that
\begin{align*}
N_{\max}\int_{\CC^n}\Bigg( \frac{\mu(D(z, 2r))}{1+|z|^s}\Bigg)^p dV(z)\geq\sum_{k=1}^\infty \int_{D(z_k,r)} \Bigg( \frac{\mu(D(z, 2r))}{(1+|z|)^s}\Bigg)^p dV(z).
\end{align*}
Now for each $z\in D(z_k, r),$ we deduce from triangle inequality that  $\mu(D(z,2r)) \geq \mu(D(z_k,r))$ and hence
\begin{align*}
 \sum_{k=1}^\infty \int_{D(z_k,r)} \Bigg( \frac{\mu(D(z, 2r))}{(1+|z|)^s}\Bigg)^p dV(z)&\geq
 \sum_{k=1}^\infty \int_{D(z_k,r)} \Bigg( \frac{\mu(D(z_k, r))}{(1+|z_k|)^s}\Bigg)^p dV(z)\nonumber\\
&\gtrsim  \sum_{k=1}^\infty  \Bigg( \frac{\mu(D(z_k, r))}{(1+|z_k|)^s}\Bigg)^p
\end{align*} from which and Lemma~\ref{oneforall} again the  required  estimate follows.
\end{proof}
We remark that the norm estimates in Lemma~\ref{lemma1} are also valid for $0<p<1$. Its proof requires a bite different
approach than the one outlined above. We decided not to develop it here since we do not need such fact in our
consideration.

\begin{lem}\label{uniformconv}
Let $0<q,p \leq\infty$ and $(g,\psi)$ be a pair of entire
functions. Then
 $uC_\psi: \mathcal{F}_{(\alpha,m)}^p \to
\mathcal{F}_{(\alpha,m)}^q $ is compact if and only if $\|
uC_\psi f_k\|_{(q,m)} \to 0$ as $k\to \infty$ for each bounded
sequence $(f_k)_{k\in \N}$ in  $\mathcal{F}_{(\alpha,m)}^p$ converging
to zero uniformly on compact subsets of $\CC^n$ as $k\to \infty.$
\end{lem}
The lemma can be proved following standard arguments; see also \cite[Lemma~8]{SS}.  The lemma will  be used repeatedly in  the proofs of
our compactness results.
 \section{Proof of the main results}
\emph{Proof of Theorem~\ref{carlesonmeasure}.}
 The equivalencies  of (ii), (iii), and (iv) come from Lemma~\ref{lemma1}. We now proceed to show that statements (iii) follows from  (i) and (i) follows from (iv).
 Assume that $\mu$ is a $(p,q)$ Fock--Carleson measure and consider a test function $K_{w}(z)=e^{\alpha \langle z, w\rangle}$ in $\mathcal{F}_{(m,p)}$. Note that this is  the kernel  function for $\mathcal{F}_{(0,\alpha)}^2$ at the point $w$.  Then
\begin{align} \label{first}
\int_{\CC^n}\frac{|K_{w}(z)|^q}{ e^{\frac{\alpha q}{2}|z|^2}} d\mu(z)\leq \|\mu\|^q \Bigg(\int_{\CC^n} \big|K_{w}(z) e^{-\frac{\alpha}{2}|z|^2}|z|^m\big|^p dV(z)\Bigg)^{\frac{q}{p}}\ \ \ \ \ \ \ \ \ \nonumber\\
\lesssim\|\mu\|^q \Bigg(\int_{\CC^n} \big|K_{w}(z) e^{-\frac{\alpha}{2}|z|^2}(1+|z|)^m\big|^p dV(z)\Bigg)^{\frac{q}{p}}. \ \ \ \ \
\end{align} By  Lemma~20 of \cite{RCKZ} or from a simple computation, the right-hand side integral  is estimated  as
\begin{align}
\label{second}
\|\mu\|^q \Bigg(\int_{\CC^n} \frac{|K_{w}(z)|^p}{(1+|z|)^{-pm}} e^{-\frac{\alpha p}{2}|z|^2} dV(z)\Bigg)^{\frac{q}{p}} \lesssim
\|\mu\|^q \bigg( \frac{e^{\frac{p\alpha}{2}|w|^2}}{(1+|w|)^{-mp}}\bigg)^{\frac{q}{p}}\nonumber\\
=\|\mu\|^q (1+|w|)^{mq} e^{\frac{q\alpha}{2}|w|^2}.\ \ \ \ \ \
\end{align} On the other hand, completing the square in the exponent on the left hand side of
\eqref{first}, we obtain
\begin{align}
\label{third}
\int_{\CC^n} \big| e^{\alpha\langle z,w\rangle-\frac{\alpha}{2}|z|^2}\big|^q d\mu(z)= e^{\frac{q\alpha}{2}|w|^2}\int_{\CC^n} e^{-\frac{q\alpha}{2}|z-w|^2} d\mu(z)\ \ \ \ \ \ \ \ \ \ \ \ \ \ \ \ \nonumber\\
\geq e^{\frac{q\alpha}{2}|w|^2}\int_{D(w,r)} e^{-\frac{q\alpha}{2}|z-w|^2} d\mu(z)\nonumber\\
\geq e^{\frac{q\alpha}{2}\big(|w|^2-r^2\big)} \mu(D(w,r))
\end{align} for all $w \in \CC^n$. Combining this with \eqref{second} leads to  (iii) and
\begin{equation}
\label{onee}
\|\mu\|^q \gtrsim \big\|\mu_{(mq, r, D)} \big\|_{L^\infty}.
\end{equation}
We next show statement  that (i) follows from (iv). The covering property of the sequence of balls $(D(z_k,r))_k$ implies
\begin{align*}
\int_{\CC^n} |f(z)|^qe^{-\frac{\alpha q}{2}|z|^2}d\mu(z) \leq \sum_{k=1}^\infty \int_{D(z_k,2r)} |f(z)|^qe^{-\frac{\alpha q}{2}|z|^2}d\mu(z)
\end{align*} Using the estimate in \eqref{estimate}, the sum is comparable to
\begin{align*}
\sum_{k=1}^\infty (1+|z_k|)^{-mq}\int_{D(z_k,2r)} |f(z)(1+|z|)^m|^qe^{-\frac{\alpha q}{2}|z|^2}d\mu(z)
\end{align*}
 which is bounded by
 \begin{align*}
 \sum_{k=1}^\infty \frac{\mu(D(z_k,2r))}{(1+|z_k|)^{mq}} \bigg(\sup_{z\in D(z_k,r))} |f(z)(1+|z|)^m|^pe^{-\frac{\alpha p}{2}|z|^2}\bigg)^{\frac{q}{p}}\ \ \ \ \ \ \ \ \ \ \ \ \ \ \ \ \ \ \ \   \\
 \lesssim\sup_{k\geq1} \frac{\mu(D(z_k,2r))}{(1+|z_k|)^{mq}}\Bigg( \sum_{k=1}^\infty  \int_{D(z_k, 3r)} \frac{ |f(z)(1+|z|)^m|^p}{e^{\frac{\alpha p}{2}|z|^2}} dV(z)\Bigg)^{\frac{q}{p}}
 =:\mathcal{S}_1.
  \end{align*}
 Now we claim  that  for each $f\in \mathcal{F}_{(m,\alpha)}^p,$
 \begin{align*}
 \sum_{k=1}^\infty\int_{D(z_k, 3r)} \frac{|f(z)(1+|z|)^m|^p}{e^{\frac{\alpha p}{2}|z|^2}} dV(z) \lesssim \sum_{k=1}^\infty\int_{D(z_k, 3r)} \frac{|f(z)|z|^m|^p}{e^{\frac{\alpha p}{2}|z|^2}} dV(z).
 \end{align*} Because of \eqref{estimate} the claim trivially follows if $|z_k|\geq 1$ for all $k.$ On the other hand, since
 $(z_k)$ assumed to be a fixed $r/2$ lattice for $\CC^n,$ its covering property ensures that the inequality $|z_k|<1$ can  happen for only a finite number of indices $k$. Thus  there exists a positive constant $M$ for which
 \begin{align*}
 \sum_{|z_k|<1}  \int_{D(z_k, 3r)} |f(z)|^pe^{-\frac{\alpha p}{2}|z|^2} dV(z)\leq M \sum_{|z_k|<1}  \int_{D(z_k, 3r)} \frac{|f(z)|^p |z|^{mp}}{e^{\frac{\alpha p}{2}|z|^2}} dV(z)\\
 \lesssim \sum_{k=1}^\infty\int_{D(z_k, 3r)} |f(z)|z|^m|^pe^{-\frac{\alpha p}{2}|z|^2} dV(z).
 \end{align*}
 Observe that the analysis above in general implies
 \begin{align}
 \label{furcomp}
 \int_{\CC^n}|f(z)(1+|z|)^m|^pe^{-\frac{\alpha p}{2}|z|^2} dV(z) \simeq  \int_{\CC^n}|f(z)|z|^m|^pe^{-\frac{\alpha p}{2}|z|^2} dV(z).
 \end{align}
  Making use of this estimate, we obtain
 \begin{align}
 \mathcal{S}_1 \lesssim\sup_{k\geq1} \frac{\mu(D(z_k,2r))}{(1+|z_k|)^{mq}}\|f\|_{(p,m)}^q
 \end{align} from which and Lemma~\ref{oneforall}, the statement in  (i) and the estimate
\begin{equation}
\label{two}
\|\mu\|^q \lesssim \big\|(1+|z_k|)^{-mq} \mu(D(z_k,r)) \big\|_{\ell^\infty}.
\end{equation} follow.
To this end, the series of norm estimates in \eqref{normestimate1} follows from
\eqref{keylem}, \eqref{onee} and \eqref{two}.

\emph{Proof of Theorem~\ref{compactcarleson}.}
The equivalency of the statements in (ii), (iii) and (iv) follows easily from a simple modification
of the proof of Lemma~\ref{gen}. Thus, we shall prove only (i) implies (iii) and (iv) implies (i). To prove the first, we consider a sequence of test functions $\xi_{(w,m)}$ defined by
 \begin{align*}\xi_{(w,m)}(z)= \frac{ k_{w}(z)}{(1+|w|)^{m}}= \frac{e^{\alpha\langle z,w\rangle-\frac{\alpha}{2}|w|^2}}{(1+|w|)^m}\end{align*}
for each  $w\in \CC^n.$ By  Lemma~20 of \cite{RCKZ} again, we have  $$\sup_{w\in \CC^n}\|\xi_{(w, m)}\|_{(p,m)}<\infty$$
 for all $p>0.$ It is also easily seen that
$\xi_{(w,m)} \to 0$ as $|w| \to \infty$, and the convergence is uniform on compact subsets of $\CC^n.$ If $\mu$ is a  $(p, q)$ vanishing Fock--Carleson measure, then
\begin{equation*}
\lim_{|w| \to \infty}\int_{\CC^n} |\xi_{(w,m)}|^q e^{-\frac{\alpha q}{2}|z|^2} d\mu(z)=\lim_{|w| \to \infty}\int_{\CC^n}\frac{e^{-\frac{\alpha q}{2}|z-w|^2}}{(1+|w|)^{mq}}  d\mu(z)= 0,
\end{equation*} from which  we have
\begin{equation*}
0\geq \lim_{|w| \to \infty}\int_{D(w,r)}  \frac{e^{-\frac{\alpha q}{2}|z-w|^2}}{(1+|w|)^{mq}}  d\mu(z)
\geq e^{-\frac{\alpha r^2}{2}}\lim_{|w| \to \infty} \frac{\mu(D(w,r)) }{(1+|w|)^{mq}}.
\end{equation*} Since  the factor $e^{-\alpha r^2/2} $ is independent of $w,$ the desired conclusion follows.

We now prove (iv) implies (i). Let $f_j$ be a sequence in $\mathcal{F}_{(m,\alpha)}^p$ such that $\sup_j \|f_j\|_{(p,m)}< \infty$
and $f_j$ converges to zero uniformly on each compact subset of $\CC^n$ as $j \to \infty.$ We aim to show that
\begin{equation*}
\int_{\CC^n} \big| f_j(z)e^{-\frac{\alpha}{2}|z|^2}\big|^q d\mu(z) \to 0
\end{equation*} as $j\to \infty$. By hypothesis, for each $\epsilon >0,$ there exists a positive integer $N_0$ such that
$\mu_{(mq,r, D)}(z_k)< \epsilon$ whenever $k\geq N_0.$ Let $U_{o}$ denotes the union of the closure of the balls $D(z_k, 3r), k=1,...N_0$. Then
$U_{o}$ is a  compact subset of $\CC^n$.  Since  $f_j$ converges to zero uniformly on compact
subsets of $\CC^n,$ there also exists $N_1> N_0$ such that
\begin{align}
\label{mig}
Q_1=\sum_{k=1}^{N_{0}}\mu_{(mq,r,D)}(z_k)\Bigg(\int_{D(z_k,3r)} (1+|z|)^{pm} |f_j(z)|^p e^{-\frac{\alpha p}{2}|z|^2}\Bigg)^{\frac{q}{p}}\ \ \ \ \ \ \ \ \ \ \ \  \nonumber\\
\leq\epsilon  \Bigg(\sum_{k=1}^{N_{0}} \int_{D(z_k,3r)} |f_j(z)|^p \frac{ (1+|z|)^{pm}}{ e^{\frac{\alpha p}{2}|z|^2}}\Bigg)^{\frac{q}{p}}
\lesssim \epsilon  \big(\sup_{z\in U_{0}} |f_j(z)|^p\big)^{\frac{q}{p}} \lesssim \epsilon
\end{align} for  all $j\geq N_1$. On the other hand,  for all $k\geq N_0$
\begin{align}
\label{yon}
Q_2=\sum_{N_{0}+1}^{\infty}\mu_{(mq,r, D)}(z_k)\bigg(\int_{D(z_k,3r)} (1+|z|)^{mp} |f_j(z)|^p e^{-\frac{\alpha p}{2}|z|^2}\bigg)^{\frac{q}{p}}\ \ \ \ \ \ \ \nonumber\\
\leq\epsilon  \bigg(\sum_{N_{0}+1}^{\infty}\int_{D(z_k,3r)} (1+|z|)^{mp} |f_j(z)|^p e^{-\frac{\alpha p}{2}|z|^2}\bigg)^{\frac{q}{p}}\nonumber\\
\lesssim \epsilon\|f_j\|_{(p,m)}^q \lesssim \epsilon \ \ \ \ \ \ \ \ \ \ \ \ \ \ \ \ \ \ \ \ \ \ \ \ \ \ \ \ \ \ \
\end{align}
Thus, using \eqref{estimate}, \eqref{mig}, \eqref{yon}, and  sufficiently large $j\geq \max\{N_0, N_1\}$,
\begin{align*}
\int_{\CC^n} \big| f_j(z)e^{-\frac{\alpha}{2}|z|^2}\big|^q d\mu(z)\leq \sum_{k=1}^\infty \int_{D(z_k,r)} |f_j(z)|^qe^{-\frac{\alpha q}{2}|z|^2} d\mu(z) \ \ \ \ \ \ \ \ \ \ \ \ \\
\lesssim \sum_{k=1}^\infty \mu_{(mq,r, D)}(z_k) \sup_{z\in D(z_k,r)}\frac{ |f_j(z)|^q(1+|z|)^{mq}}{e^{\frac{\alpha q}{2}|z|^2}}\\
\leq Q_1+ Q_2 \lesssim \epsilon.
\end{align*}
\emph{Proof of Theorem~\ref{strongcarlesonmeasure}.}
Since $p/(p-q)\geq 1,$ by  Lemma~\ref{lemma1}, (iii), (iv) and (v) are again  equivalent. To show that   statement (v) follows from (i) we may
follow the classical Leuecking's approach via  Khinchine's equality in \cite{DL}. Consider a function $f$ in  $\mathcal{F}_{(m,\alpha)}^p.$  Then by Lemma~\ref{Fourier},  $ z^\beta f$ belongs to $ \mathcal{F}_{\alpha}^p$ for all multi-indices $\beta$ such that $\beta_1+ \beta_2 +...\beta_n= m$. Thus, there  exists a sequence  $c_j \in \ell^p,  0<p\leq \infty$,  for which
\begin{equation} \label{special}
z^\beta f(z)= \sum_{j=1}^\infty c_j k_{z_j}(z) \in \mathcal{F}_{(\alpha)}^p\ \text{and }\ \ \|f\|_{(p,m)} \simeq \||z|^m k_{z_j}\|_{p}  \lesssim \|(c_j)\|_{\ell^p.}
\end{equation} This was proved  in  \cite{SJR} for $p\geq 1$ and in \cite{RW} for $0<p<1$.
We first assume that $0<q<\infty$.  Since $\mu$ is a
$(p, q)$ Fock--Carleson measure,
\begin{equation*}
\int_{\CC^n} |f(z)|^q e^{-\frac{\alpha q}{2}|z|^2} d\mu(z) \leq
\|\mu\|^q \|f_c\|_{(p,m )}^q \lesssim \|\mu\|^q
\|(c_j)\|_{\ell^p}^q.
\end{equation*}  If $(r_j)$ is  the Rademacher
sequence of functions on $[0,1]$ chosen  as in \cite{DL}, then
Khinchine's inequality yields
\begin{equation}\Bigg(\sum_{j=1}^\infty
|c_j|^2|k_{z_j}(z)|z|^{-m}|^2\Bigg)^{q/2} \lesssim \int_0^1
\Bigg|\sum_{j=1}^\infty c_jr_jk_{z_j}(z)|z|^{-m}\Bigg|^qdt.
\label{Khinchine}
\end{equation}
Note that here if the $r_j$ are chosen as refereed above, then
$(c_jr_j)\in \ell^p$ with
$\|(c_jr_j)\|_{\ell^p} = \|(c_j)\|_{\ell^p}$ and
 \begin{equation*}
 \sum_{j=1}^\infty c_jr_jk_{z_j}(z)z^{-\beta} \in \mathcal{F}_{(m,\alpha)}^p, \ \text{with}\
 \Big\|\sum_{j=1}^\infty c_jr_j(t)k_{(z_j,\alpha)}(z)z^{-\beta} \Big\|_{(p, m)} \lesssim \|(c_j)\|_{\ell^p}
 \end{equation*} for all multi-indices $\beta$ such that $\beta_{sn}= m.$
Making use of first \eqref{Khinchine} and subsequently Fubini's
theorem, we obtain
\begin{align}
\label{combine}
\int_{\CC^n} \Bigg( |c_j|^2|k_{z_j)}(z)|z|^{-m}|^2\Bigg)^{q/2} d\mu(z) \ \ \ \ \ \ \ \ \ \ \ \ \ \ \ \ \ \ \ \
\ \ \ \ \ \ \ \ \ \ \ \ \ \ \ \ \ \ \ \ \ \ \ \ \ \ \nonumber \\
 \lesssim \int_{\CC^n}\Bigg(\int_0^1 \bigg| \sum_{j=1}^\infty c_j r_j(t)k_{z_j}(z)|z|^{-m}\bigg|^q
dt\Bigg) d\mu(z)\nonumber\\
=\int_0^1\Bigg(\int_{\CC^n}\bigg| \sum_{j=1}^\infty c_j
r_j(t)k_{z_j}(z)|z|^{-m}\bigg|^q d\mu(z)\Bigg)dt\nonumber\\
\lesssim\|\mu\|^q \|(c_j)\|_{\ell^p}^q. \ \ \ \ \ \ \ \ \ \ \ \ \ \  \ \ \ \ \ \ \ \ \ \ \ \ \ \ \ \ \ \ \ \ \ \ \
\end{align}
Now if $q\geq 2,$ then using \eqref{estimate}, we have
\begin{align}
\label{partly}
\sum_{j=1}^\infty |c_j|^q \frac{\mu(D(z_j,r))}{(1+|z_j|)^{mq}}=
\int_{\CC^n}\sum_{j=1}^\infty |c_j|^q \frac{\chi_{D(z_j,r)}(z)}{(1+|z|)^{mq}}
d\mu(z)\\
\leq \int_{\CC^n}\Bigg(\sum_{j=1}^\infty |c_j|^2
\frac{\chi_{D(z_j,r)}(z)}{(1+|z|)^{2m}}\Bigg)^{q/2}d\mu(z)\ \ \ \ \
 \label{pbig}
\end{align}
where the last inequality is since $q/2 \geq 1$ and $|c_j|\geq 0$ for all $j.$ On the other hand, if
$q<2,$ then applying H\"{o}lder's inequality with exponent $2/q$ to the integral in \eqref{partly} gives
\begin{align*}
\int_{\CC^n}\sum_{j=1}^\infty |c_j|^q\frac{\chi_{D(z_j,r)}(z)}{(1+|z_j|)^{mq}}  d\mu(z)
&\leq N_{\max}^{\frac{2-q}{2}}\int_{\CC^n}\Bigg(\sum_{j=1}^\infty |c_j|^2
\frac{\chi_{D(z_j,r)}(z)}{(1+|z|)^{2m}}\Bigg)^{q/2}d\mu(z)\\
&\lesssim \int_{\CC^n}\Bigg(\sum_{j=1}^\infty |c_j|^2
\frac{\chi_{D(z_j,r)}(z)}{(1+|z|)^{2m}}\Bigg)^{q/2}d\mu(z).
\end{align*}
The last integral here  and in \eqref{pbig} are  bounded by
\begin{align*}
 e^{\frac{q\alpha }{2}r^2}\int_{\CC^n}\Bigg(\sum_{j=1}^\infty |c_j|^2
\frac{e^{-\alpha|z-z_j|^2}}{(1+|z|)^{2m}}\Bigg)^{q/2}d\mu(z)\lesssim
\int_{\CC^n} \bigg( |c_j|^2|k_{z_j)}(z)|z|^{-m}|^2\bigg)^{q/2}d\mu(z).
\end{align*}This combined with \eqref{combine} gives
\begin{equation}
\label{dual} \sum_{j=1}^\infty |c_j|^q\frac{\mu(D(z_j,r))}{(1+|z_j|)^{mq}} \lesssim
\|\mu\|^q\|(c_j)\|_{\ell^p}^q= \|\mu\|^q \Big(\sum_{j=1}^\infty |c_j|^p\Big)^{q/p}.
\end{equation} Then a duality argument gives that
\begin{equation}
\label{three1}
\mu_{(mq,r,D)}(z_j) \in \ell^{\frac{p}{p-q}}\ \ \text{and}\ \ \|\mu\|^q \gtrsim \| \mu_{(mq,r,D)}(z_j) \|_{\ell^{\frac{p}{p-q}}}.
\end{equation}
We now prove (iv) implies (i).  Integrating both side of \eqref{est} against the measure $\mu$ and subsequently using
$ \chi_{D(z,r)}(w)= \chi_{D(w,r)}(z)$, and Fubini's theorem we get
\begin{align}
\int_{\CC^n}|f(w)e^{-\frac{\alpha}{2}|w|^2}|^q d\mu(w)\lesssim \int_{\CC^n}\big|f(w)e^{-\frac{\alpha}{2}|w|^2}\big|^q \mu(D(w,r))dV(w)\nonumber\\
= \int_{\CC^n}\big|f(w)e^{-\frac{\alpha}{2}|w|^2}(1+|w|)^m\big|^q \mu_{(mq,r,D)}dV(w)
\label{need}
\end{align}
Applying  H\"{o}lder's inequality with exponent $p/q$ and \eqref{furcomp}
\begin{align*}
\label{confus}
\int_{\CC^n}  \big|f(w)e^{-\frac{\alpha}{2}|w|^2} (1+|w|)^m\big|^q \mu_{(mq,r,D)dV(w)}\lesssim \|f\|_{(p,m)}^q \big\|\mu_{(mq,r,D)}\big\|_{L^{\frac{p}{p-q}}}.
\end{align*}
 It follows from this
that the estimate
\begin{equation*}
\label{three2}
\|\mu\|^q \lesssim \|\mu_{(mq,r,D)}\|_{L^{\frac{p}{p-q}}}
\end{equation*}  which together with  \eqref{three1} and \eqref{keylem} yields  the series of norm estimates in \eqref{normestimate2}.

Obviously, (ii) implies (i). We proceed to show its converse. Let $f_j$ be a sequence of
functions in $\mathcal{F}_{(m,\alpha)}^p$ such that $\sup_j
\|f_j\|_{(p,m)}<\infty$ and $f_j$ converges uniformly to
zero on compact subsets of $\CC^n$ as $j\to \infty.$ For a fixed
$R>\delta >0,$ we write
\begin{align*}
  \int_{\CC^n}|f_j(z)|^qe^{\frac{-q\alpha}{2}|z|^2}
d\mu(z)&= \bigg(\int_{|z|\leq R-\delta} +\int_{|z|>R-\delta}\bigg)
|f_j(z)|^q
e^{-\frac{q \alpha}{2}|z|^2}  d\mu(z)\nonumber\\
&=I_{j1} +I_{j2}.
\end{align*} We estimate the two pieces of integrals independently and consider first $I_{j1}.$ Since
$f_j \to 0$ uniformly on compact subsets of $\CC^n$ as $j\to
\infty$, we find
\begin{align*}
\limsup_{j\to \infty}I_{j1}&=\limsup_{j \to \infty}\int_{|z|\leq
R-\delta} |f_j(z)|^q
e^{-\frac{q\alpha}{2}|z|^2}  d\mu(z)\nonumber\\
&\leq\limsup_{j \to \infty} \sup_{|z|\leq
R-\delta}|f_j(z)|^q\int_{|z|\leq R-\delta} e^{-\frac{q\alpha}{2}|z|^2} d\mu(z)\nonumber\\
 &\lesssim\limsup_{j \to \infty} \sup_{|z|\leq
R-\delta}|f_j(z)|^q \to 0, \ \text{as}\  \ j\to \infty.
\end{align*}
If we denote by $\mu^R$ the truncation of  $\mu$ on the set
$\{z\in\CC^n: |z|>R-\delta\},$ then applying \eqref{need} we obtain,
\begin{align*}
\limsup_{j\to \infty}I_{j2}&=\limsup_{j \to\infty}\int_{|z|>R-\delta} |f_j(z)|^q
e^{-\frac{q\alpha}{2}|z|^2}d\mu(z)\\
&=\limsup_{j \to \infty}\int_{\CC^n} |f_j(z)|^qe^{-\frac{q\alpha}{2}|z|^2}d\mu^R(z) \\
&\lesssim\limsup_{j \to \infty}\int_{\CC^n}
|f_j(z)|^q ( 1+|z|)^{mq} e^{-\frac{q\alpha}{2}|z|^2} \mu^R_{(mq, r, D)}dV(z).
\end{align*}
 Applying H\"{o}lder's inequality again, we obtain
\begin{align*}
\limsup_{j \to \infty}\int_{\CC^n}
|f_j(z)|^q ( 1+|z|)^{mq} e^{-\frac{q\alpha}{2}|z|^2}\mu^R_{(mq, r, D)}dV(z)
\ \ \ \ \ \ \ \ \ \ \ \ \ \ \ \ \ \ \ \ \ \ \ \ \ \ \ \ \ \ \ \\
\leq \limsup_{j \to \infty}\|f_j\|_{(p,m)}^q\int_{\CC^n}
\Big|\mu^R_{(mq, r, D)}\Big|^{\frac{p}{p-q}}dV(z)  \ \ \ \ \ \ \\
=\limsup_{j \to \infty} \|f_j\|_{(p, m)}^q
\int_{|z|>R-r} \Big|\mu^R_{(mq, r, D)}\Big|^{\frac{p}{p-q}}dV(z).
\end{align*}
Since $ \sup_j \|f_j\|_{(p,m)}<\infty \ \ \text{and}\ \ \mu^R_{(mq, r, D)}\in L^{\frac{p}{p-q}},$
 we let $R \to \infty$ in the above relation to conclude that $\mu $ is a  $(p,q)$ vanishing Fock--Carelson measure, and completes the proof of the theorem.

\emph{Proof of Theorem~\ref{boundedinfinity}.}
The proof of the theorem closely follows the arguments used  in the proof of Theorem~\ref{strongcarlesonmeasure}. We
will sketch only some of the required modifications  below. The equivalencies  of the statements  in  (iii), (iv) and (v) follow from
Lemma~\ref{lemma1} with $s= mq$. We observe that the  global  geometric condition (vi) follows from
(iii) when we in particular set  $t=1$. Because by Fubini's theorem, we may have
 \begin{align}
 \label{tesfaa}
   \int_{\CC^n} \widetilde{\mu}_{(1,mq)}(z) dV(z)&=\int_{\CC^n}\int_{\CC^n} \frac{ e^{\frac{\alpha}{2} |\langle w, z\rangle|^2 -\frac{\alpha}{2}
 |z|^2-\frac{\alpha}{2} |w|^2} }{(1+|z|)^{mq}}d\mu(w)dV(z)\nonumber\\
 &=\int_{\CC^n} \bigg(\int_{\CC^n} \frac{e^{-\frac{\alpha}{2}
 |z-w|^2}}{(1+|z|)^{mq}} dV(z)\bigg)d\mu(w).
  \end{align} Since $$(1+|z|)^{-1} \leq \frac{1+ |z-w|}{1+|w|}, \ z, \ w \in \CC^n,$$ the integral in \eqref{tesfaa}
  is bounded by
  \begin{align}
  \label{furth}
  \int_{\CC^n} (1+|w|)^{-mq} \int_{\CC^n} (1+ |z-w|)^{mq} \frac{e^{-\frac{\alpha}{2}
 |z-w|^2}}{(1+|z|)^{mq}} dV(z)d\mu(w)\ \ \ \ \  \nonumber\\
 \lesssim \int_{\CC^n} \frac{1}{(1+|w|)^{mq}} d\mu(w)= \mu_{mq}(\CC^n).
  \end{align}
  On the other hand, an application of \eqref{estimate} gives
  \begin{align}
  \label{further}
  \int_{\CC^n} \bigg(\int_{\CC^n} \frac{e^{-\frac{\alpha}{2}
 |z-w|^2}}{(1+|z|)^{mq}} dV(z)\bigg)d\mu(w) &\geq \int_{\CC^n} \int_{D(w,1)} \frac{e^{-\frac{\alpha}{2}
 |z-w|^2}}{(1+|z|)^{mq}} dV(z)d\mu(w)\nonumber\\
 &\gtrsim e^{-\alpha/2} \int_{\CC^n} \frac{1}{(1+|w|)^{mq}} d\mu(w)\nonumber\\
 &\simeq \mu_{mq}(\CC^n).
  \end{align} Combining \eqref{tesfaa}, \eqref{furth}, and \eqref{further} we obtain
  \begin{align}
  \label{tesf}
   \int_{\CC^n} \widetilde{\mu}_{(1,mq)}(z) dV(z)\simeq \mu_{mq} (\CC^n).
  \end{align} This shows that shows that (vi) holds if and only
 if (iii) holds for $t=1$.

We now prove (i) implies (v). For this, we simply modify the proof of  (i) implies (v) in the proof of  Theorem~\ref{strongcarlesonmeasure}. Thus,  replace $p$
by $\infty$ and follow the same arguments until we get  equation \eqref{dual} which would be in this case
\begin{equation}
\label{con} \sum_{j=1}^\infty |c_j|^q \mu_{(mq, r, D)}(z_j) \lesssim
\|\mu\|^q\|(c_j)\|_{\ell^\infty}^q.
\end{equation}
Since $(c_j)$ is an arbitrary sequence
in $\ell^\infty,$ we may in particular  set $c_j=1$  for all $j$ in
the above relation to make the desired conclusion. Observe that this
particular choice in \eqref{con} also ensures
\begin{equation}
\label{oneside}\big\|  \mu_{(mq, r, D)}(z_j)\big\|_{\ell^1} \lesssim
\|\mu\|^q.
\end{equation}
To prove that (i) follows from  (iii), observe that
applying  \eqref{need} to a function
  $f$ in $ \mathcal{F}_{(m,\alpha)}^\infty$ gives
\begin{align}
\int_{\CC^n}  \big|f(w)e^{-\frac{\alpha}{2}|w|^2}\big|^q d\mu(w) &\lesssim
\int_{\CC^n}  \big|f(w)e^{-\frac{\alpha}{2}|w|^2}\big|^q (1+|w|)^{mq}\mu_{(mq,r,D)}(w)dV(w)\nonumber\\
 &\leq\|f\|_{(\infty,m)}^q\int_{\CC^n}   \mu_{(mq,r,D)}(w)dV(w)\nonumber\\
 &=\|f\|_{(\infty,m)}^q\|\mu_{(mp,r,D)}\|_{L^1}
\label{book1}
\end{align} which  completes the proof for (iii) implies (i).
From  \eqref{book1}, we also have
\begin{equation}
\label{normpart} \|u\| \lesssim
\|\mu_{(mp,r,D)}\|_{L^1}^{1/q}
\end{equation} from which, \eqref{oneside}, \eqref{tesf} and \eqref{keylem}, the series of  norm estimates
in \eqref{fourequal} follow.

It remains to show (ii) follows from (i). But this can be easily  done  by  simply modifying
a similar proof in  Theorem~\ref{strongcarlesonmeasure}. Thus, we omit the details.

\emph{Proof of Theorems~\ref{bounded},\ \ref{bounded1} and \ref{bounded3}}.
The central  idea in these proofs is to translate the given problem into a $(p,q)$ embedding map problem for the
Fock--Sobolev spaces; through which we may  invoke the notion of $(p,q)$ Fock--Carleson measures and apply the results already proved  in the preceding parts.

 For each $p>0,$ we set $\theta_{(m, p)}$ to be the
positive pull back  measure on $\CC^n$ defined by \begin{equation*}
\theta_{(m,  p)}(E)=  \int_{\psi^{-1}(E)}
|u(z)|^p|z|^{mp} e^{-\frac{\alpha p}{2}|z|^2}
dV(z)\end{equation*} for every Borel  subset $E$ of $\CC^n$.
Then by substitution, we have
\begin{align*}
\|uC_{\psi)}f\|_{(m, q)}^q&\simeq \int_{\CC^n} |f(z)|^q
d\theta_{(m,q)}(z)=\int_{\CC^n} |f(z)e^{-\frac{\alpha}{2}|z|^2}|^q  e^{\frac{q\alpha}{2}|z|^2}
d\theta_{(m,q)}(z)\\
&=\int_{\CC^n} |f(z)e^{-\frac{\alpha}{2}|z|^2}|^q d\lambda_{(m,q)}(z)
\end{align*} where $d\lambda_{(m,q)}(z)= e^{\frac{q\alpha}{2}|z|^2}
d\theta_{(m,q)}(z).$  This shows that $ uC_{\psi}: \mathcal{F}_{(m,\alpha)}^p \to \mathcal{F}_{(m,\alpha)}^q$   is bounded if and only if $\lambda_{(m,q)}$ is a $(p, q)$ Fock--Carleson measure. We may now consider three different cases depending on the size of the exponents.

\textbf{Case 1: $p \leq q$.} In this case, by Theorem~\ref{carlesonmeasure}, the boundedness of $ uC_{\psi}$ holds if and only if
$\widetilde {\lambda}_{(m,q)}$ belongs to $L^\infty$.  But substituting back  $d\lambda_{(m,q)}$ and $ d\theta_{(m,q)}$ in terms of $dV $ results in
\begin{align*}
\widetilde {\lambda}_{(m,q)}(z)&= \int_{\CC^n} \frac{ e^{-\frac{q\alpha}{2}|w-z|^2}}{(1+|w|)^{mq}} d\lambda_{(m,q)}(w)\\
 &= \int_{\CC^n} \frac{|k_{z}(\psi(w))|^q e^{-\frac{q\alpha}{2}|w|^2}}{(1+|\psi(w)|)^{mq}}|u(w)|^q|w|^{mq} dV(w)\\
 &=B_{(m, \psi)}(|u|^q)(z).
 \end{align*} The norm estimate in \eqref{norm1estimate} easily follows from  the series of norm  estimates in
 Theorem~\ref{carlesonmeasure}.

The proof of part (ii) of   Theorem~\ref{bounded} is  similar to the first part. This time we  need  to argue with  Theorem~\ref{compactcarleson}  instead of Theorem~\ref{carlesonmeasure}. Thus, we omit the trivial details.

\textbf{Case 2:  $0<q<p<\infty.$} By Theorem~\ref{strongcarlesonmeasure},  $\lambda_{(m,q)}$ is a $(p, q)$ Fock--Carleson measure if and only if $\lambda_{(m,q)}$ is a  $(p, q)$ vanishing Fock--Carleson measure. This  again holds if and only if  $\widetilde {\lambda}_{(m,q)}=B_{(m, \psi)}(|u|^q)$ belongs to $L^{p/(p-q)}$. The norm estimate in \eqref{normless1} also follows from the series of norm estimates in
\eqref{normestimate2}.

\textbf{Case $3: 0<q<\infty$  and $p=\infty$}. As in the previous cases, by Theorem~\ref{boundedinfinity}, $\lambda_{(m,q)}$ is an $(\infty, q)$ Fock--Carleson measure if and only if  $\lambda_{(m,q)}$ is an  $(\infty, q)$ vanishing Fock--Carleson measure which is equivalent to the fact that  $\widetilde {\lambda}_{(m,q)}=B_{(m, \psi)}(|u|^q)\in L^1$. The norm estimate in \eqref{normless2} follows again  from the estimates in \eqref{fourequal}.

\emph{Proof of Theorem~\ref{bounded2}.}
We first note if  $m=0,$  the function
\begin{align*}B^\infty_{(m,\psi)} (|u|)(z)
 = \frac{|z|^m |u(z)|}{(1+|\psi(z)|)^m} e^{\frac{\alpha}{2}\big(|\psi(z)|^2-|z|^2\big)}= |u(z)| e^{\frac{\alpha}{2}\big(|\psi(z)|^2-|z|^2\big)},
  \end{align*}and for this particular case, the theorem was proved in \cite{SS}. We now generalize the proof
for any $m.$
From a  simple application of Lemma~3 of \cite{RCKZ}, we conclude
\begin{equation}
\label{appl}
|f(z)| \leq \frac{\|f\|_{(p,m)}}{(1+|z|)^{m}} e^{\frac{\alpha}{2}|z|^2}
\end{equation} for each $f$ in $\mathcal{F}_{(m,\alpha)}^p$ and $0<p\leq \infty.$ This implies
\begin{align*}
\|uC_\psi f\|_{(\infty,m)}&= \sup_{z\in \CC^n} |u(z)||z|^m |f(\psi(z))|e^{\frac{-\alpha}{2}|z|^2}\\
&\leq \|f\|_{(p,m)} \sup_{z\in \CC^n}  \frac{|u(z)||z|^m}{(1+|\psi(z)|)^{m}} e^{\frac{\alpha}{2}|\psi(z)|^2-\frac{\alpha}{2}|z|^2}\\
&=\|f\|_{(p,m)} \sup_{z\in \CC^n} B^\infty_{(m,\psi)} (|u|)(z)
\end{align*} from which one side of the estimate in \eqref{normbound},
\begin{equation}
\|uC_\psi\| \leq  \|B^\infty_{(m,\psi)} (|u|)\|_{L^\infty},
\end{equation}
 and the sufficiency of part (i) of the theorem follow.

 To prove the necessity part of the theorem,   for  each point  $w\in \CC^n$  we use again  the  sequence of
test functions  $\xi_{(w, m)}(z)= (1+|w|)^{-m} k_{w}(z).$ Then
\begin{equation}\label{normb} \|\xi_{(w,m)}\|_{(p,m)} \lesssim 1 \end{equation}  independent of
$p$ and $w$ which follows by Lemma~20 of \cite{RCKZ} for $p<\infty$ and  from a simple argument for $p= \infty.$
Applying $uC_\psi$ to $\xi_{(w, m)}$ and completing the square on the exponent yields
\begin{align*}
\|uC_\psi\| \gtrsim \| uC_\psi \xi_{(w, m)} \|_{(\infty, m)}\geq  \frac{|u(z)||z|^m}{(1+|w|)^{m}} e^{\frac{\alpha}{2}\big(|\psi(z)|^2 - |\psi(z)-w|^2-|z|^2\big)}
\end{align*}
for all points $w$ and $z$ in $\CC^n.$ Setting $w= \psi(z)$ in particular leads to
\begin{equation*}
\|uC_\psi\| \gtrsim  \frac{|u(z)||z|^m}{(1+|\psi(z)|)^{m}} e^{\frac{\alpha}{2}|\psi(z)|^2 -\frac{\alpha}{2}|z|^2}=  B^\infty_{(m,\psi)} (|u|)(z)
\end{equation*} from which the necessity of the condition  and   the remaining side of the estimate in \eqref{normbound} follow.

To prove the second part of the theorem, we first assume that $uC_\psi$ is compact. The  sequence  $\xi_{(w, m)}$ converges to zero as
$|w| \to \infty,$ and the convergence is uniform on compact subset of $\CC^n.$  We further assume that there exists sequence of points
$z_j \in \CC^n$ such that $|\psi(z_j)| \to \infty$ as $j\to \infty.$ If such a sequence does not exist, then \eqref{normcomp} holds trivially. It follows from compactness of  $uC_\psi$ that
\begin{equation}
\limsup_{j\to \infty} B_{(m,\psi)}^\infty (|u|)(z_j)  \leq \limsup_{j\to \infty}\| uC_\psi \xi_{(\psi(z_j)), m} \|_{(\infty,m)}=0
\end{equation} from which \eqref{normcomp} follows.

We next suppose that $uC_\psi$ is bounded and  condition \eqref{normcomp}holds. We proceed to show compactness of
      $uC_\psi $. The condition along with Theorem~\ref{bounded} implies that  $uC_\psi $ is a bounded map.  On the other hand, the  function $f(z)= 1$ belongs to $\mathcal{F}_{(p,m)}$ , in deed, a  computation along \eqref{normequal} results in  $\|f\|_{(p,m)}=1$. It follows that by  boundedness, the weight function $u$ belongs to $\mathcal{F}_{(m,\alpha)}^\infty$.
Let $f_j$ be a sequence of functions in $\mathcal{F}_{(m,\alpha)}^p$ such that $\sup_j
\|f_j\|_{(p,m)}<\infty$ and $f_j$ converges uniformly to
zero on compact subsets of $\CC^n$ as $j\to \infty.$  For each $\epsilon >0$  by  \eqref{normcomp} there  exists
a positive $N_1$ such that
\begin{equation*}
 B_{(m,\psi)}^\infty(|u|) (z) <\epsilon
\end{equation*} for all $|\psi(z)| > N_1.$ From this together  and \eqref{appl}, we obtain
\begin{align*}
 |uC_\psi f_j(z) ||z|^m e^{-\frac{\alpha}{2}|z|^2}&=
 |u(z) f_j(\psi(z))||z|^m e^{-\frac{\alpha}{2}|z|^2}\\
  &\leq \|f_j\|_{(p,m)}\frac{ |u(z)||z|^m}{(1+|\psi(z)|)^{m}} e^{\frac{\alpha}{2}|\psi(z)|^2-\frac{\alpha}{2}|z|^2}
  \lesssim \epsilon
 \end{align*} for all $|\psi(z)| > N_1$ and all $j.$ On the other hand if $|\psi(z)|\leq N_1,$ then it easily seen that
\begin{align*}
  |u(z) f_j(\psi(z))||z|^m e^{-\frac{\alpha}{2}|z|^2}
  &\leq\|u\|_{(\infty,m)}\sup_{z:|\psi(z)|\leq N_1 } |f_j(\psi(z))|\\
  &\lesssim  \sup_{z:|\psi(z)|\leq N_1 } |f_j(\psi(z))|  \to 0
   \end{align*} as $j\to \infty$, and completes the proof.

   \emph{Proof of Corollary~\ref{cor2}.}
   We first assume that  $uC_{\psi}:
\mathcal{F}_{(m,\alpha)}^p  (\ \text{or}\ \mathcal{F}_{(0,m,\alpha)}^\infty) \to
\mathcal{F}_{(0,m,\alpha)}^\infty$ is compact  and aim to verify condition \eqref{limit}. It follows that
 $uC_{\psi}:
\mathcal{F}_{(m,\alpha)}^p  (\ \text{or}\ \mathcal{F}_{(0,m,\alpha)}^\infty) \to
\mathcal{F}_{(m,\alpha)}^\infty$ is also compact. Then by part (ii) of Theorem~\ref{bounded2}, for each $\epsilon,$ there exists
a positive integer $N_1$ such that
\begin{equation*}
 B_{(m,\psi)}^\infty (|u|)(z) <\epsilon
\end{equation*} for all $|\psi(z)| > N_1.$  On the other hand, setting $f(z)=1,$ by boundedness (which follows from compactness) we have
$ u\in \mathcal{F}_{(0,m,\alpha)}^\infty$. Thus, there  exists a positive integer $N_2$  for which
\begin{equation*}
 |u(z)||z|^m e^{-\frac{\alpha}{2}|z|^2} <\epsilon e^{-\frac{\alpha}{2}N_1^2}
\end{equation*} for all $|z| > N_2.$  Therefore,   if $|z|>N_2>N_1$ we have
\begin{equation*}
 B_{(m,\psi)}^\infty (|u|)(z)\leq   e^{\frac{\alpha N_1^2}{2}}|u(z)|z|^m e^{\frac{-\alpha}{2}|z|^2} <\epsilon
\end{equation*} as desired.

For the converse, let $f_j$ be a uniformly bounded  sequence of functions in $\mathcal{F}_{(m,\alpha)}^p  (\ \text{or}\ \mathcal{F}_{(0,m,\alpha)}^\infty)$ which  converges uniformly to
zero on compact subsets of $\CC^n$ as $j\to \infty.$  For each $\epsilon >0,$ condition \eqref{limit} implies that there exists
a positive integer  $N_3$  for which
$
 B_{(m,\psi)}^\infty (|u|)(z) <\epsilon
$ for all $|z| > N_3.$  From this and  \eqref{appl}, we obtain
\begin{equation*}
 |uC_\psi f_j(z) ||z|^m e^{-\frac{\alpha}{2}|z|^2}
  \leq \sup_{j\geq1}\|f_j\|_{(p,m)}\frac{ |u(z)||z|^m}{(1+|\psi(z)|)^{m}} e^{\frac{\alpha}{2}|\psi(z)|^2-\frac{\alpha}{2}|z|^2}
  \lesssim \epsilon
 \end{equation*} for all $|z| > N_3$  and  all exponents $p$. On the other hand, since the set  $\{|\psi(z)|: |z|\leq N_3 \}$ is compact, there exists a positive integer  $N_4$  for which  $\sup_{ |z|\leq N_3 }|\psi(z)| \leq N_4$. Thus,
  \begin{align*}
  |u(z) f_j(\psi(z))||z|^m e^{-\frac{\alpha}{2}|z|^2}
  &\leq\|u\|_{(\infty,m)}\sup_{|w|\leq N_4 } |f_j(w)|\\
  &\lesssim  \sup_{|w|\leq N_4 } |f_j(w)|  \to 0, \ j\to \infty.
   \end{align*}
 \textit{Proof of Theorem~\ref{essentialnorm}.}
The proof of the theorem follows a classical approach  used to prove similar results
 in \cite{ZZH, TM1, SS,  UEKI1,UEKI, UEKI2}.
 Recall that each entire function $f$ can be expressed as $f(z)=
\sum_{k=0}^\infty p_k(z)$ where the function  $p_k$ are  polynomials of
degree $k.$ We consider  a sequence of operators $R_j$  defined by
\begin{equation}
\label{series}
(R_j f)(z)= \sum_{k=j}^\infty
p_k(z).
\end{equation}
It was proved in \cite{DGP,UEKI1} that
\begin{align}
\label{series}
\lim_{j\to \infty} \|R_jf\|_{p} = 0
\end{align}
for each $f$ in the ordinary Fock spaces $\mathcal{F}_\alpha^p$, and $1<p<\infty.$ Thus by the uniform boundedness principle
\begin{align} \label{uniform}
\sup_{j\geq 1}\|R_j\|<\infty.\end{align}
Now if $h\in \mathcal{F}_{(m,\alpha)}^p,$ then by Lemma~\ref{Fourier} and \eqref{series}
\begin{align}
\lim_{j\to \infty}\|R_jh\|_{(p,m)}\simeq\lim_{j\to \infty}\|R_j(z^\beta h)\|_p= 0
\end{align} from which the same conclusion \eqref{uniform} follows when the sequence $(R_n)$ is defined on
 weighted Fock--Sobolev space.

 Let $1<p\leq q\leq \infty$ and assume that    $uC_\psi: \mathcal{F}_{(m,\alpha)}^p \to  \mathcal{F}_{(m,\alpha)}^q$  is bounded. Then mimicking
 the proof of Lemma~2 in  \cite{UEKI2} yields
 \begin{equation}
 \label{uppessential}
\|uC_\psi \|_e \leq \liminf_{j \to \infty} \| uC_\psi
R_j\|_{(q,m)}.
\end{equation}
Having singled out this important inequality, we now proceed to prove the lower estimates in
the theorem. To this end, let  $Q$ be a
compact operator acting between   $\mathcal{F}_{(m,\alpha)}^p$ and $\mathcal{F}_{(m,\alpha)}^q$.  We first suppose that $q= \infty$. Since
 $\xi_{(w,m)}$ converges
to zero uniformly on compact subset of $\CC^n$ as $|w| \to \infty$ and \eqref{normb} holds, we
have
\begin{align}
\label{firstcase}
\|uC_\psi-Q\| &\geq\limsup_{|w| \to \infty}
\|uC_\psi \xi_{(w,m)}-Q
\xi_{(w,m)}\|_{(\infty,m)}\nonumber \\
&\geq \limsup_{|w| \to
\infty}\|uC_\psi \xi_{(w,m)}\|_{(\infty,m)}-\|Q
\xi_{(w,m)}\|_{(\infty,m)}\nonumber\\
&=\limsup_{|w| \to
 \infty}\|uC_\psi \xi_{(w,m)}\|_{(\infty,m)}\nonumber\\
 &\geq \limsup_{|\psi(w)| \to
\infty} B^\infty_{(\psi,m)}(|u|)(z)(w),
 \end{align}
 where the first equality is due to compactness of $Q.$

For $0<q<\infty,$ we consider a different sequence   of test functions in
$\mathcal{F}_{(m,\alpha)}^P$, namely that$ k_w;$ the normalized reproducing kernel function in $\mathcal{F}_{(0,\alpha)}^2$. This sequence replaces the role played
by $\xi_{(w,m)}$  above and running the same  procedure  as in \eqref{firstcase}  gives
\begin{align*}
\|uC_\psi-Q\| &\geq \limsup_{|w| \to
\infty}\|uC_\psi k_{w}\|_{(q,m)}-\|Q
\varphi_{w}\|_{(q,m)}\\
&=\limsup_{|w| \to
 \infty}\|uC_\psi k_{w}\|_{(q,m)}\\
 &\geq \limsup_{|w| \to
 \infty}\int_{\CC^n}\frac{|u(z)|^q |z|^{mq}}{(1+|\psi|)^{mq}} |k_w(\psi(z))|^{q} e^{-\frac{\alpha q}{2}|z|^2} dV\\
 &=\bigg(\limsup_{|w| \to
\infty} B_{(\psi,m)}(|u|^q)(w) \bigg)^{\frac{1}{q}}.
  \end{align*}
    From this  and \eqref{firstcase} the lower
 estimate in \eqref{essential} follows.

  To prove the upper estimate,  we again consider the next two different cases.\\
\textbf{Case 1}: Suppose $q<\infty.$ Then for each $f$ of  unit norm  in $\mathcal{F}_{(m,\alpha)}^p$,
 we get
\begin{align}
\label{firstintegral}
 \|uC_\psi R_j f\|_{(m, q)}^q\simeq   \int_{\CC^n} |R_jf(z)|^q d\theta_{(m, q)}(z)\ \ \ \ \ \ \ \ \ \ \ \ \ \   \ \ \ \ \nonumber\\
 =\bigg(\int_{ \CC^n\setminus D(0,\delta)}+
\int_{ D(0,\delta)}\bigg)|R_j f(z)|^qe^{-\frac{\alpha q}{2}|z|^2} d\lambda_{(m,q)}(z)
\end{align} where again $d\lambda_{(m,q)}(z)= e^{\frac{q\alpha}{2}|z|^2}
d\theta_{(m,q)}(z)$ and for some fixed $ \delta>0.$   By Theorem~\ref{bounded}, the first integral in \eqref{firstintegral} is bounded
by
\begin{align*}
\|R_jf\|_{(p,m)}^q \Big(\sup_{z\in \CC^n\setminus D(0,\delta)} B_{(\psi,m)}(|u|^q)(z) \Big)
 \lesssim \sup_{z\in \CC^n\setminus D(0,\delta)} B_{(\psi,m)}(|u|^q)(z)
\end{align*} where we used the fact that  $\sup_j\|R_j\|<\infty.$ It remains to estimate the second integral in
\eqref{firstintegral}. Again by Theorem~\ref{bounded} and followed by the n-variable version of  Lemma~3 in \cite{UEKI2}, the integral is estimated
as
\begin{align*}
\int_{ D(0,\delta)}|R_j f(z)|^qe^{-\frac{\alpha q}{2}|z|^2}  d\lambda_{(m,q)}(z) \ \ \ \ \ \ \ \ \ \ \ \ \ \ \ \ \ \ \ \ \ \ \ \ \ \ \ \ \ \ \ \ \ \ \ \ \ \ \ \ \ \ \ \ \ \ \ \ \ \ \ \ \  \\
\lesssim \sup_{z\in \CC^n} B_{(\psi,m)}(|u|^q)(z)\int_ { D(0,\delta)}|z|^{mp} | R_j f(z)|^pe^{-\frac{\alpha p}{2}|z|^2} dV(z)\\
\lesssim  \sup_{z\in \CC^n} B_{(\psi,m)}(|u|^q)(z)I_{j}
\int_{\CC^n}e^{-\frac{p\alpha}{2}|z|^2} dV(z)\ \ \ \ \ \ \ \ \ \ \ \\
\lesssim \sup_{z\in \CC^n} B_{(\psi,m)}(|u|^q)(z)I_{j} \ \ \ \ \ \ \ \ \
\end{align*} where
\begin{equation}
\label{neww}
 I_{j}\simeq \Bigg(\sum_{m= j}^\infty
 (\delta\alpha)^m \sum_{\beta_{ns}= m} (\beta!)^{-1} \prod_{l=1}^n \Big( \frac{2}{\alpha s}\Big)^{\frac{\beta_l}{2}+\frac{1}{s}} \Big(\Gamma\Big( \frac{s\beta_l}{2}+1\Big)\Big)^{\frac{1}{s}}\Bigg)^p
\end{equation} with $s$ the conjugate exponent of $p$ and $\beta!=\prod_{l=1}^n \beta_l! $.  Observe that by
Stirling's approximation formula, we have
\begin{align}
\Big(\Gamma\Big( \frac{s\beta_l}{2}+1\Big)\Big)^{\frac{1}{s}}\simeq \Big(\frac{\beta_l s}{2}\Big)^{\frac{\beta_l}{2}+\frac{1}{s}-\frac{1}{2s}} e^{-\frac{\beta_l}{2}}.
\end{align} Plugging this in \eqref{neww} and applying the ration test it is easily seen that  the series
converges and hence $I_{j}\to 0$ as $j\to \infty$. Thus, the  contribution from  the second integral in
\eqref{firstintegral} goes to zero for large enough $j$.
 Therefore
 \begin{equation*}
 \lim_{j\to \infty} \sup_{\|f\|_{(p,m)}= 1}
\|uC_\psi R_j f\|_{(q,m)}^q\lesssim \sup_{z\in \CC^n\setminus
D(0,\delta)} B_{(\psi,m)}(|g|^q)(z).\end{equation*} By
\eqref{uppessential} we get \begin{equation*} \|uC_\psi\|_e^q
\lesssim \lim_{\delta \to \infty}\sup_{ z\in \CC^n\setminus D(0,\delta)}
B_{(\psi,m)}(|u|^q)(z)\simeq \limsup_{|z| \to \infty}
B_{(\psi,m)}(|u|^q)(z)
\end{equation*} and completes the proof for the first case.

\textbf{Case 2: $q= \infty.$ } Not much effort is needed to prove this case since it follows by a simple modification of
the arguments used in the previous case. We shall
sketch it out for simplicity of the exposition. Acting similarly as above, for each $f$ of  unit norm in $\mathcal{F}_{(m,\alpha)}^p$, we may invoke \eqref{appl} to  get
\begin{eqnarray*}
|z|^m |uC_\psi R_jf(z)| e^{-\frac{\alpha}{2}|z|^2}&=&  |u(z)||z|^m |R_jf(\psi(z))|e^{\frac{-\alpha}{2}|z|^2}\\
&\leq& \|R_jf\|_{(p,m)}   \frac{|u(z)||z|^m}{(1+|\psi(z)|)^{m}} e^{\frac{\alpha}{2}|\psi(z)|^2-\frac{\alpha}{2}|z|^2}\\
&\leq& \sup_{j\geq 1} \|R_j\| B^\infty_{(m,\psi)} (|u|)(z) \lesssim B^\infty_{(m,\psi)} (|u|)(z)
\end{eqnarray*} from which we have that
\begin{equation*}
\sup_{|\psi(z)|\geq \delta} |z|^m |uC_\psi R_jf(z)| e^{-\frac{\alpha}{2}|z|^2}\lesssim \sup_{|\psi(z)|\geq \delta} B^\infty_{(m,\psi)} (|u|)(z).
\end{equation*}On the other hand, since $uC_\psi$ is bounded, the weight function $u$ belongs to $\mathcal{F}_{(m,\alpha)}^\infty.$ Thus by  Lemma~3 in
\cite{UEKI2} again, we have
\begin{eqnarray*}
\sup_{|\psi(z)|<\delta}|z|^m |uC_\psi R_nf(z) e^{-\frac{\alpha}{2}|z|^2}|\leq \|u\|_{(\infty,m)}I_{j}
\end{eqnarray*} where $I_j$  and $s$ are as in \eqref{neww}.
 But it is again easily seen that $I_j \to 0$ as $j \to \infty.$ Therefore,
\begin{equation*}
\liminf_{j\to \infty} \sup_{\|f\|_{(p,m)}= 1} \ \ \sup_{|\psi(z)|\leq \delta} |uC_\psi R_j f(z)||z|^m e^{-\frac{\alpha}{2}|z|^2} =0
\end{equation*} and hence
\begin{equation*}
\|uC_\psi\|_e \lesssim \sup_{|\psi(z)|\geq \delta} B^\infty_{(m,\psi)} (|u|)(z)
\end{equation*} from which we get
\begin{equation*}
\|uC_\psi\|_e \lesssim \limsup_{|\psi(z)|\to \infty }B^\infty_{(m,\psi)} (|u|)(z)
\end{equation*} after letting $\delta$ to $\infty,$ and completes
the proof.

\end{document}